\documentclass[11pt, a4paper]{amsart}
\usepackage{amsmath,amsthm,amssymb}
\usepackage{graphicx}
\usepackage{mathtools}
\usepackage{times}
\usepackage{enumerate}
\usepackage{algorithm}
\usepackage{algpseudocode}
\usepackage{soul}
\usepackage{booktabs}
\floatevery{algorithm}{\setlength\hsize{8cm}}
\usepackage[usenames,dvipsnames]{color}
\usepackage[british]{babel}

\theoremstyle{definition}
\newtheorem{theorem}{Theorem}[section]
\newtheorem{corollary}[theorem]{Corollary}
\newtheorem{proposition}[theorem]{Proposition}
\newtheorem{lemma}[theorem]{Lemma}

\newtheorem{remark}[theorem]{Remark}
\newtheorem{example}[theorem]{Example}

\numberwithin{equation}{section}

\usepackage{amssymb}
\usepackage{amsbsy}
\usepackage[pdftex,bookmarks=true,pdfborder={0 0 0}]{hyperref}
\usepackage{epstopdf}
\usepackage{bbold}

\newcommand\R{\mathbb{R}}

\newcommand\Z{\mathbb{Z}}
\renewcommand\S{\mathbb{S}}

\newcommand{\id}{\operatorname{id}}

\newcommand{\Id}{{\mathbb{1}}}
\newcommand{\Sing}{\operatorname{sing}}
\newcommand{\Dist}{\operatorname{dist}}
\newcommand{\diam}{\operatorname{diam}}

\newcommand{\M}{\mathcal{M}}

\newcommand{\D}{\mathcal{D}}
\renewcommand{\P}{\text{P}}
\newcommand{\Q}{\text{Q}}

\newcommand{\DiOs}{{\mathcal{\tilde D}}_{tg}}
\newcommand{\Di}[1]{\mathcal{D}_{\text{#1}}}

\newcommand{\SO}{\text{SO}}
\renewcommand{\O}{\text{O}}
\newcommand{\SPD}{\text{SPD}}

\newcommand{\dist}{{\bf{d}}}
\newcommand{\n}{{\bf{n}}}
\renewcommand{\v}{{\bf{v}}}
\newcommand{\tr}{{\mathrm{tr}}}

\newcommand{\W}{W}

\newcommand{\Cof}{{\operatorname{Cof}}}

\newcommand{\vol}{\operatorname{vol}}
\newcommand{\supp}{\operatorname{supp}}
\newcommand{\im}{\operatorname{im}}
\newcommand{\dd}{{\operatorname{d}}}

\newcommand{\cof}{\textrm{Cof\,}}

\newcommand\restr[2]{{
\left.\kern-\nulldelimiterspace #1 \vphantom{\big|} \right|_{#2} 
}}

\newcommand{\eg}{e.g.,\ }
\newcommand{\ie}{i.e.,\ }

\newcommand{\etal}{et al.\ }

\newcommand{\notinclude}[1]{}

\newcommand{\added}[2][]{%
  \ifthenelse{\equal{#1}{}}{{#2}}{{#2}}%
}

\begin{document}

\title{Shape Aware Matching of Implicit Surfaces based on Thin Shell Energies}

\author{Jos\'{e} A. Iglesias}
\address[Jos\'{e} A. Iglesias]{Computational Science Center, University of Vienna\\
Oskar-Morgenstern-Platz 1, 1090 Vienna, Austria}
\email{jose.iglesias@univie.ac.at}

\author{Martin Rumpf}
\address[Martin Rumpf]{Institute for Numerical Simulation, Universit\"{a}t Bonn\\
Endenicher Allee 60, 53115 Bonn, Germany}
\email{martin.rumpf@uni-bonn.de}

\author{Otmar Scherzer}
\address[Otmar Scherzer]{Computational Science Center, University of Vienna\\
Oskar-Morgenstern-Platz 1, 1090 Vienna, Austria; RICAM, Austrian Academy of Sciences\\
Altenberger Str. 69, 4040 Linz, Austria}
\email{otmar.scherzer@univie.ac.at}

\begin{abstract}
A shape sensitive, variational approach for the matching of surfaces considered as thin elastic shells is investigated.
The elasticity functional to be minimized takes into account two different types 
of nonlinear energies: 
a membrane energy measuring the rate of tangential distortion
when deforming the reference shell into the template shell, 
and a bending energy measuring the bending under the deformation in terms of the change of the shape 
operators from the undeformed into the deformed configuration.
The variational method applies to surfaces described as level sets.
It is mathematically well-posed and an existence proof of an optimal matching deformation is given.
The variational model is implemented using a finite element discretization combined with a narrow band approach on an efficient hierarchical grid structure.  For the optimization a regularized nonlinear conjugate gradient scheme and a cascadic multilevel strategy are used. 
The features of the proposed approach are studied for synthetic test cases and a collection of geometry processing examples.
\end{abstract}



\maketitle

\section{Introduction}
\label{sec:intro}
We present a variational model for the matching of surfaces implicitly represented as level sets. The approach is inspired by the mathematical theory of nonlinear elasticity
of thin shells. The model consists in an energy functional, which is to be minimized among deformations of a computational domain in which two given surfaces are embedded. A minimizer of this functional is a deformation that closely maps one (reference) surface onto the other (template) surface. As the underlying model we consider the reference surface as a thin elastic shell, \ie a layer of an elastic material embedded in a volume of another several orders of magnitude softer isotropic elastic material. Subject to matching forces the volume is deformed in such a way that the thin shell is mapped onto the template surface.
The functional reflects desired phenomena like resistance to compression and expansion of the surface, resistance to bending, and rotational invariance, while solely involving  the deformation and the Jacobian of the deformation.  
The model is formulated in terms of projected derivatives from the tangent space of the reference surface onto the expected tangent space of the template surface.
Taking into account a suitable factorization of the natural pullback under a deformation of shape operators enables us to formulate a model with appropriate convexity properties. The actual surface matching constraint is handled through a penalty, allowing for efficient numerical computation.

Through arguments of compensated compactness, we are able to show weak lower semicontinuity of the energy and consequently existence of minimizing deformations. 
We present a numerical approach based on a multilinear finite element ansatz for the deformation implemented on adaptive octree grids. The resulting discrete energy is minimized in a multiscale fashion applying a regularized gradient descent.

In the conference article \cite{IgBeRuSc13} a preliminary version of this approach was presented. For the functional in that paper lower semicontinuity could not be ensured for either the membrane or bending energies. This lack of lower semicontinuity manifests itself in applications, where compression of the surface is expected, and leads to undesired oscillations in almost-minimizing deformations, which we explore in the present work through explicit examples and computations. Additionally, to increase the efficiency the computational meshes are in the present paper adapted to the surfaces. Consequently the number of degrees of freedom scales asymptotically almost like that of a surface problem.

The main pillar of our modelling is the use of polyconvex energy densities, first introduced in \cite{Bal77}. Energies of this type allow for geometric consistency properties like rotation invariance and the ability to measure area and volume changes. The core insight of this theory is that integrands consisting of convex functions of subdeterminants of the Jacobian give rise to integral functionals that are weakly lower semicontinuous in suitable Sobolev spaces. Indeed, this can be seen as an instance of compensated compactness \cite{Mur87}. A generic polyconvex isotropic energy density of the type used in this work is\begin{equation}\label{eq:genericDensity}
\alpha_p \|A\|^p + \beta_q \|\cof A\|^q + \Gamma(\det A),
\end{equation}
for $\cof A \added{:= \det A\, A^{-T}}$ the cofactor matrix of $A$. Here, the coefficients and the function $\Gamma$ are such that \eqref{eq:genericDensity} attains \added{a local} minimum for $A \in \SO(n)$, \added{that is, for rigid motions. Such an example is provided below in \eqref{def:What}}. Often in the modelling of nonlinear elasticity the condition 
\begin{equation}\label{eq:nonDegCondition}
\lim_{\det A \to 0^+ }\Gamma(A) = +\infty
\end{equation}
is added, to reflect the non-interpenetration of matter \cite{Bal81}. In our model we make use of densities both with and without this property.

\paragraph*{\bf Related work.}
Linear elasticity has been extensively used in computer vision and in graphics. Prominent applications are image registration \cite{Mod04, KybUns03, RisPetSam10, KabFraFis06, KabLor10}, optical flow extraction \cite{KeRi05}, and shape modeling \cite{FucJueSchYan09a}. Recently, theories of nonlinear elasticity have been applied in many computer vision and graphics problems such as mesh deformation \cite{ChaPinSanSch10}, shape averaging  \cite{RuWi08},  registration of medical images \cite{BurModRut13}. 
The advantage of nonlinear models is that they allow for intuitive deformations when the displacements are large.

In this paper, we present a model for nonlinear elastic matching of thin shells.
A finite element method for the discretization of bending energies of biological membranes has been introduced in \cite{BoNoPa10}. Their approach uses quadratic isoparametric finite elements to approximate the interface on which the gradient flow of an elastic energy of Helfrich type is considered. The papers \cite{BrePocWir13, BrePocWir15} discuss accurate convex relaxation of higher order variational problems on curves described as jump sets of functions of bounded variation. In particular, it enables the numerical treatment of elastic energies on such curves.

One challenge in polyhedral surface processing is to provide consistent notions of curvatures and second fundamental forms, \ie notions that converge (in an appropriate topology or in a measure theoretic sense) to their smooth counterparts, given a smooth limit surface.
One computationally popular model for discretizing the second fundamental form is Grinspun's \etal \emph{discrete shells} model~\cite{GriHirDesSch03}.
Another efficient, and robust method for nonlinear surface deformation and shape matching is PriMo~\cite{BotPauGroKob06}. This approach is based on replacing the triangles of a polyhedral surfaces by thin prisms. During a deformation, these prisms are required to stay rigid, while nonlinear elastic forces are acting between neighboring prisms to account for bending, twisting, and stretching of the surface. We refer to Botsch and Sorkine~\cite{BotSor08} for a discussion of pros and cons for various such methods. In comparison with methods based on polyhedral surfaces, level set approaches like ours are not dependent on specific triangulations of the shapes.

The matching of surfaces with elastic energies has recently been studied in \cite{WiScSc11}.
Their energy contains a membrane energy depending on the Cauchy-Green strain tensor and a bending-type energy comparing the mean curvatures on the surfaces.
The matching problem is formulated in terms of a binary linear program in the product space of sets of surface patches. For computations, a relaxation approach is used.

A different direction is the use of parametric approaches to reduce shape matching problems to the matching of functions on a fixed domain. For example, the methods presented in \cite{ZhaHer99} and \cite{WanEtal06} are based on conformal maps from the unit disk. A more general variant using conformal maps on surfaces with arbitrary topology is presented in \cite{LiEtal08}. Within the family of parametric methods, a surface matching approach related to ours is presented in \cite{DrLiRuSc05}, where nonlinear elastic energies are used for matching parametrized surface patches. In comparison to all these methods, our level set approach is non-parametric and allows surfaces of any topology, which does not need to be fixed in advance.

In \cite{SrSaJo09}, face matching based on a matching of corresponding level set curves on the facial surfaces is investigated. To match pairs of curves an optimal deformation between them is computed using an elastic shape analysis of curves. Compared to our approach, this model does not take into account bending dissipation of the curves.

A different direction in shape recognition and matching is exploiting the intrinsic geometry of the surfaces only, thereby producing isometry-invariant methods based on the first fundamental form, like those in \cite{ElaKim03, BroBroKim08}. In comparison, bending is penalized in our model and we use all curvatures of the surfaces and their directions to be able to better match regions of edges and creases correctly.

A method for matching and blending of curves represented by level sets has been presented in \cite{MuRa12}. Thereby, a level set evolution generates an interpolating family of curves, where the associated propagation speed of the level sets depends on differences of level set curvatures. In this class of approaches, geometric evolution problems are formulated, whereas here we focus on variational models for matching deformations. Variational registration of implicit surfaces was also considered in \cite{LeeLai08}, but only through volume elasticity, in contrast to our shell terms.

To summarize, the main novelty of our contribution is the combination of independence of mesh topologies arising from the use of level sets, penalization of tangential distortion in a rotationally-invariant framework, and awareness both of curvatures and curvature directions of the surfaces in the matching. We are not aware of any other methods possessing all of these features simultaneously.

Our approach is inspired by the articles \cite{DelZol94, DelZol95} in which surface PDE models are derived in terms of the signed distance function. Shape warping based on the framework of \cite{DelZol94} has been discussed from a geometric perspective in \cite{CharFauFer05}.
\paragraph*{\bf Outline.}
The paper is organized as follows. In Section 2, we review the required preliminaries about distance functions and formulate the geometric non-distortion and matching conditions that inspire our model. In Section 3, we present the different contributions to our energy. Section 4 is devoted to proving the existence of minimizing deformations under suitable Dirichlet and Neumann boundary conditions. Furthermore, the strong convergence of solutions for vanishing matching penalty parameter is discussed and counterexamples showing the lack of lower semicontinuity of related simpler models are given. In Section 5 a numerical strategy for minimizing the energy on adaptive octree grids is presented. Finally, Section 6 contains a range of numerical examples demonstrating the behaviour of solutions corresponding to our design criteria, and presents several potential applications.

\paragraph*{\bf Some useful notation.}
For later usage and the purpose of reference let us collect  some useful notation, mostly introduced in detail in later sections:
\begin{itemize}
\item $|B|$ stands for the Lebesgue measure of $B \subset \R^n$, and $\diam B = \sup_{x,y \in B}|x-y|$ for its diameter.
\item Generic matrices are denoted by $A, B, M, N$.  We use $\Id$ for the identity matrix. 
The set of rotations is denoted by $\O(n)$ and $\SO(n)$ is the set of 
orientation-preserving rotations. The set of all symmetric and positive definite matrices is
$\SPD(n)$. 
\item Components of vectors are denoted with subindices. 
For  $v\in \R^n$, $|v|$ denotes its Euclidean norm. The $(n-1)$-dimensional sphere is $\S^{n-1}$. 
For a matrix $M$, $|M|$ is the Frobenius norm.
\item For two column vectors $v,w \in \R^n$, $v \otimes w$ is the tensor product of $v$ and $w$, that is, the square matrix $v w^T$. 
In particular, if $|w|=1$ we have the identity $(v \otimes w) w = v$.
\item  $\P(e)=\Id - e \otimes e$ is the projection onto vectors orthogonal to $e \in \S^{n-1}$.
\item Deformations on $\R^n$ are denoted by $\phi$, and deformations defined on a \added{hyper}surface $\M \subset \R^n$ by $\varphi$. The identity deformation is \added{denoted by }$\id$.
\item $\Omega \subset \R^n$ denotes the computational domain. Every relevant deformation $\phi$ maps $\Omega$ into $\R^n$.
$\Omega$ has to contain all computationally relevant manifolds $\M$. $\Omega$ has Lipschitz boundary, is open and bounded.
\item We use the notation $\partial_i$ for partial derivatives, $\nabla$ for the gradient of a scalar function, $\D$ for the 
Jacobian matrix of a vector function and $\D^2$ for the Hessian matrix of a scalar function.
\item  $\M_1, \M_2 \subset \Omega$ are $C^{2,1}$ compact hypersurfaces. 
          The inside and outside components of $\Omega \setminus \M_i$ are well defined by the Jordan-Brouwer 
          separation theorem (\cite{GuiPol74}, Chapter 2, Section 5).
          
          The signed distance function to $\M_1, \M_2$ is denoted by $\dist_1, \dist_2$. 
          The sign convention is that $\dist_i$ is negative on the inside of $\M_i$, so that $\dist_i(x)=-\Dist(x, \M_i)$ if $x$ is in the inside component of $\Omega \setminus \M_i$ and $\dist_i(x)=\Dist(x, \M_i)$ otherwise,
          where the distance functions $\Dist(\cdot, \M_i)$, $i=1,2$ are the unique viscosity solutions of $1-|\nabla \Dist(\cdot, \M_i)|=0$ and $\Dist(\cdot, \M_i)=0$ on $\M_i$.
          The normal fields to the offsets of $\M_i$ at a point $x$ are denoted by $\n_i(x):=\nabla \dist_i(x)$.
          A superscript next to $\M_i$ ($i=1,2$), as in $\M_i^c$, denotes that we are talking about a level set of $\dist_i$ with value 
          different from zero, so that $\M_i^c := \dist_i^{-1}(c)$.

          $T_x\M_i^{\dist_i(x)}$ denotes the tangent space to $\M_i^{\dist_i(x)}$ at $x$.
          The outwards normal to $\M_i^{\dist_i(x)}$ is given by $\n_i(x)$, and the set of points where 
          $\dist_i$ is not differentiable is denoted by $\Sing \dist_i$. 

          We use $\mathcal{S}_i = \D^2 \dist_i$ for the Hessian of $\dist_i$, which coincides with an extended shape operator of $\M_i$.
\item $\lambda, \mu$ are the  Lam\'{e} coefficients of an isotropic material in linearized elasticity. 
\item  $C^0(\Omega;\R^n)$ is the space of continuous functions from the domain $\Omega$ to the range $\R^n$, $C^{k,\alpha}$ the H\"{o}lder spaces in which the $k$-th derivative is $\alpha$-H\"{o}lder continuous, including the Lipschitz case $\alpha=1$. The range of the spaces is specified unless it is $\R$. Sobolev spaces are denoted by $W^{1,p}$ and the closure of compactly supported smooth functions in them by $W_0^{1,p}$.
\item The letter $C$ is reserved for a generic positive constant that may have different values in each appearance.
Sequence indexing is usually denoted by a superscript $k$, and limits by an overline, \eg $\phi^k \to \overline{\phi}$.
\end{itemize}

\section{Deformation and matching of level set \added{hyper}surfaces}
\label{sec:deformation}
We are given two compact, connected embedded hypersurfaces $\M_1,\M_2$ of class $C^{2,1}$, which are diffeomorphic to each other, 
and both of which are contained in a bounded Lipschitz domain $\Omega \subset \R^{n}$. 
In this section we deal with the tangential distortion and the change of the shape operator under a deformation $\phi: \Omega \to \R^n$.

For any $c \in \R$, we denote the $c$-offsets to the \added{hyper}surface $\M_i$ by
$\M_i^c := \{ x \in \Omega \,|\, \dist_i(x) = c \}\,$.
Furthermore, we define the singularity set $\Sing\dist_i$ as the set of points where $\dist_i$ is not \st{twice} differentiable. With the regularity of $\M_i$ that we have assumed, it is well known 
(\eg Theorem 1.1, Corollary 1.3 and Remark 1.4 of \cite{LiNir05}) that $\Sing\dist_i$ has Lebesgue measure zero and $\Dist(\M_i, \Sing \dist_i) > 0$. Furthermore, combining \cite[Theorem 5.6]{DelZol94} and \cite[Proposition 4.6, 7.]{ManMen03} we see that $\dist_i \in C^2(\overline{\Omega} \setminus \overline{\Sing\dist_i})$.

The gradient of the signed distance function $\nabla \dist_i(x)$ is the outward-pointing unit normal $\n_i(x)$ to $\M_i^{\dist_i(x)}$ at a point $x$.
The tangent space to $\M^{\dist_i(x)}_i$ at $x$, denoted by $T_x \M^{\dist_i(x)}_i$, consists of all vectors orthogonal to $\n_i(x)$. 
Then, the corresponding projection matrices onto the tangent spaces are defined by 
\[\P_i(x) := \P(\n_i(x))=\Id - \n_i(x) \otimes \n_i(x).\] 
Note that $\mathcal{S}_i(x):=\D^2 \dist_i(x)=\D \n_i(x) \P_i(x)$ is the shape operator of the immersed \added{hyper}surface $\M^{\dist_i(x)}_i$ at a point $x$.
In fact, from $|\n_i(x)|^2=1$ we deduce by differentiation that $\n_i^T(x)\mathcal{S}_i(x)=0$. This, together with 
the fact that $\n_i \otimes \n_i$ is the projection onto the normal of the \added{hyper}surface $\mathcal{S}_i$ shows that 
\begin{equation*}
\P_i(x) \D \n_i (x) = \D \n_i(x).
\end{equation*}
With our choice of signs for $\dist_i$, the symmetric matrices $\mathcal{S}_i$ are positive semidefinite for convex \added{hyper}surfaces $\M_i$. 
Further information on tangential calculus for level set functions may be found in Chapter 9 of \cite{DelZol11}.

\subsection{Tangential derivative and area and length distortion} \label{sec:tangential}
First, let us assume that $\phi$ exactly maps $\M_1^c$ onto $\M_2^c$, for all $c>0$.
Then, 
$T_x \M_1^{\dist_1(x)}=\im\,\P_1(x)$ and $T_{\phi(x)}\phi(\M_1^{\dist_1(x)})=T_{\phi(x)}\M_2^{\dist_2(\phi(x))}=\im \P_2 ( \phi(x) )$
and we define the tangential derivative induced by the deformation $\phi$ as 
\begin{equation}\label{def:Dtt}
\Di{tg}\phi(x):=\P_2( \phi(x) ) \, \D\phi(x) \, \P_1 (x)\,,
\end{equation}
capturing the tangential variation of $\phi(x)$ on $\M_2$ along tangential directions on $\M_1$. 
In the variational model we consider below an energy term depending on $\Di{tg}\phi(x)$ will reflect the tangential 
distortion of the deformation in the context of a matching of the two \added{hyper}surfaces 
$\M_1$ and $\M_2$ even though  $\phi(\M_1)$ does not necessarily equal $\M_2$.
Indeed, in the case  $\M_2 \neq \phi(\M_1)$  the variation along a tangent direction on 
$\M_1$ is still projected via $\Di{tg}\phi(x)$ onto the tangent space 
$T_{\phi(x)}\M_2^{\dist_2(\phi(x))}$ and not onto the tangent space of the deformed \added{hyper}surface 
$\phi(\M_1)$ (cf. Fig. \ref{fig:Dtt}). Therefore there may exist tangential directions
$v \in T_x \M_1^{\dist_1(x)}$, such that  $\Di{tg}\phi(x) v =0$ even though $\D\phi v \neq 0$.
Thus $\Di{tg}\phi(x)$ can only be considered a measure of tangential distortion if $\phi(\M_1)$ is 
sufficiently close to $\M_2$ in the sense of closeness of tangent bundles.

\begin{figure}[ht]
\begin{center}
 \includegraphics[width=.7\textwidth]{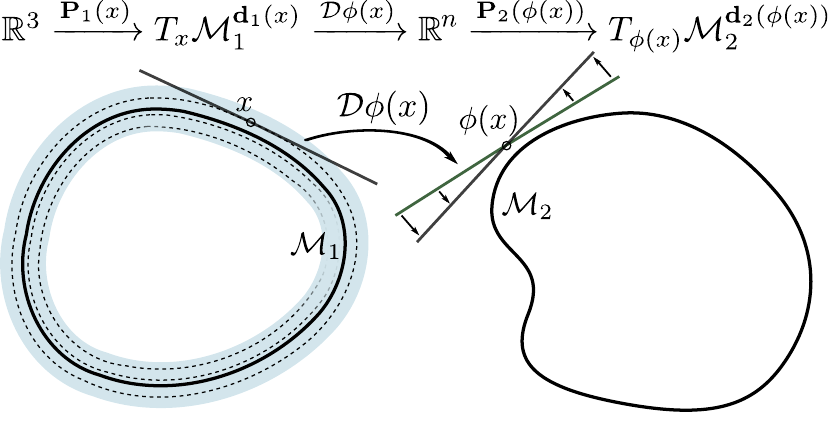}
    \caption{A sketch of the tangential derivative $\Di{tg}\phi$ in the non-exact matching case with $\phi(\M_1) \neq \M_2$.}
    \label{fig:Dtt}
    \end{center}
\end{figure}

For a general deformation $\psi:\R^n\to \R^n$ the Cauchy-Green strain tensor $\D \psi^T \D \psi$ describes (up to first order) 
the deformation in a frame invariant (with respect to rigid body motions) way. 
Since we are interested in the effect of such a deformation between two \added{hyper}surfaces, for a suitably extended tangential gradient 
$\Di{tg}\phi + \n_2\circ \phi \otimes \n_1$ we define the {\it extended tangential part of the Cauchy-Green strain tensor}, measuring only tangential distortion: 
\begin{equation}\label{eq:Ctt}
\left(\Di{tg}\phi + (\n_2\circ \phi) \otimes \n_1\right)^T \left(\Di{tg}\phi + (\n_2 \circ \phi) \otimes \n_1\right) = 
\Di{tg}\phi^T \Di{tg}\phi + \n_1  \otimes \n_1\,.
\end{equation}
The term $\n_2(\phi(x)) \otimes \n_1(x)$ is used to complement directions that are removed by the projections in the 
definition of the tangential distortion $\Di{tg}\phi$ and can be seen to realize a nonlinear Kirchhoff-Love assumption 
\cite[Page 336]{Cia00}, which postulates that lines normal to the middle surface of a shell remain normal after the deformation, without stretching.

Next, we investigate the area and length distortion due to the tangential derivative $\Di{tg}\phi$.
For a given vector $e \in \R^n$ we denote by $\Q(e)$ any proper rotation such that $\Q(e) e_n = e$, 
where $e_n$ denotes the $n$-th element of the canonical basis of $\R^n$. 
Note that this condition does not specify a unique $\Q(e)$.
Then, for every $B \in \R^{n \times n}$ satisfying $w \in \ker B$ and $\im B \subseteq v^\perp$ for some unit 
vectors $v,w \in \S^{n-1}$, we have
\begin{equation}\label{eq:complement}
\begin{gathered}
\Q(v)^T ( B + v \otimes w ) \Q(w) = \Q(v)^T B \Q(w) + e_n \otimes e_n=\left(
\begin{array}{c|c}
\tilde{B} &  0 \\ \hline
0  &  1
\end{array}
\right),
\end{gathered}
\end{equation}
where $\tilde{B}$ is the upper left $(n-1)\times(n-1)$ submatrix of $\Q(\added{v})^T B \Q(\added{w})$. 
Obviously  \eqref{eq:complement} implies
\begin{eqnarray*}
\det(B + v \otimes w)&=&\det(\tilde{B})\,,\\
|B+ v\otimes w|^2&=&\tr\big( (B + v\otimes w)^T(B+ v\otimes w) \big)= 1+ | \tilde{B} |^2\,.
\end{eqnarray*}
Hence,  for  $\phi(\M_1)=\M_2$ and  $v=n_2(\phi(x))$, $w=n_1(x)$ the area distortion under the \added{hyper}surface matching deformation $\phi$ at some position $x$ 
is described by $\det(\Di{tg}\phi(x) + \n_2(\phi(x)) \otimes \n_1(x))$, which equals the positive square root of the determinant of the above Cauchy-Green strain tensor $\Di{tg}\phi^T \Di{tg}\phi + \n_1  \otimes \n_1$.
The squared tangential length distortion (in the sense of summing all squared distortions with respect to an orthogonal basis)
is described by $|\Di{tg}\phi(x) + \n_2(\phi(x)) \otimes \n_1(x))|^2$ and equals the trace of the Cauchy-Green strain tensor.

\subsection{\added{Bending} and curvature mismatch}
Now, we quantify the change of curvature directions and magnitudes under the deformation $\phi$.
Our approach is motivated by models describing bending of elastic shells, because in our application the 
\added{hyper}surfaces are considered as thin shells.

In order to quantify the changes of curvature we first assume that $\phi(\M_1)=\M_2$,
and compute the difference of the pull back of the shape operator $\mathcal{S}_2$ on $\M_2$ onto $\M_1$ 
under the deformation $\phi$ and the shape operator $\mathcal{S}_1$ on $\M_1$, which, for two arbitrary 
directions $v,\,w\in \R^n$, is given by 
$$
\mathcal{S}_2(\phi(x)) \D \phi(x) v \cdot \D \phi(x) w -  \mathcal{S}_1(x)  v \cdot  w = 
\left(\D \phi(x)^T \mathcal{S}_2(\phi(x)) \D \phi(x) - \mathcal{S}_1(x)\right) v\cdot w\;.
$$
If $v,\,w$ are tangent vectors in $T_x \M_1$, this difference describes the {\bf relative shape operator}.

We define the {\bf extended relative shape operator}
\begin{equation}\label{eq:extshaperel}
\mathcal{S}_{rel}(x) := \D \phi(x)^T \mathcal{S}_2(\phi(x)) \D \phi(x) - \mathcal{S}_1(x)\,.
\end{equation}

For $n=3$ and when $\phi$ is an isometric deformation between $\M_1$ and $\M_2$ (that is, $\D \phi(x)$ is an 
orthogonal mapping on $T_x \M_1$ for all $x\in \M_1$), $\mathcal{S}_{rel}$ appears in physical models 
for thin elastic shells in the context of the $\Gamma-$limit of $\rm{3D}$ hyperelasticity \cite{FriJamMorMul03}.
Even though we do not necessarily expect our deformations to be tangentially isometric, we use 
this ansatz to compare curvatures of level sets in deformed and undeformed configuration, respectively. 
The following calculations shed some light on the properties of $\mathcal{S}_{rel}$:
\begin{equation}\label{eq:deformedShapeOp}
\begin{aligned}
\D^2(\dist_2 \circ \phi)(x)&=\added{\D\big((\D \phi)^T(\n_2 \circ \phi)\big)(x)}\\
&=\D\phi(x)^T \mathcal{S}_2(\phi(x)) \D\phi(x) + \sum_{k=1}^n \big(\n_2\added{(\phi(x))}\big)_k \D^2\phi^k(x)\,.
\end{aligned}
\end{equation}
The assumption that $\phi(\M_1)=\M_2$ can be rewritten as $\dist_2 \circ \phi(x) = 0$ for $x \in \M_1$.
Let us assume that in addition $\dist_2 \circ \phi$ is a distance function (that is $|\nabla ( \dist_2 \circ \phi )|=1$),
then $\dist_2 \circ \phi$ is again a distance function, and since $\dist_2 \circ \phi = 0$ it follows that 
the left hand side of \eqref{eq:deformedShapeOp} is the shape operator of the \added{hyper}surface $\M_1$.
The first term in the right hand side is the pullback of $\mathcal{S}_2$. 


Let us remark that the appearance of a second fundamental form is consistent with Koiter's nonlinear thin shell theory 
\cite{Koi66}, \cite[Section 11.1]{Cia00}.
Regardless of whether $\dist_2 \circ \phi$ is a distance function or not, \eqref{eq:deformedShapeOp} implies that 
\begin{equation}\label{eq:bendingExplanation}
\begin{aligned}
 \D \phi(x)^T \mathcal{S}_2(\phi(x)) \D \phi(x) - \mathcal{S}_1(x) &= 
- \sum_{k=1}^n (\n_2\added{(\phi(x))})_k \D^2\phi^k(x)+\D^2(\dist_2\circ \phi - \dist_1)(x) \\
&= 
- \sum_{k=1}^n (\n_2\added{(\phi(x))})_k \D^2\phi^k(x),
\end{aligned}
\end{equation}
in case $\dist_2 \circ \phi = \dist_1$. In the next section, we use the extended relative shape operator to derive a variational model for the mismatch of curvatures.

\section{Energy functional}\label{sec:levelsets}
Given two \added{hyper}surfaces $\M_1$ and $\M_2$ our ultimate goal is to describe best matching deformations $\phi$, which map $\M_1$ onto $\M_2$ as the minimizer of a 
suitable energy. Thereby, different energy terms will reflect a set of matching conditions for a volumetric deformation $\phi:\Omega \to \R^{n}$ and without a hard constraint $\phi(\M_1) = \M_2$:
\begin{itemize}
\item A membrane deformation energy $E_{\text{mem}}$ penalizes the tangential distortion measured through $\Di{tg}\phi$.
\item A bending energy $E_{\text{bend}}$ penalizes bending as reflected by the relative shape operator.
\item A matching penalty $E_{\text{match}}$ ensures a proper matching of the two \added{hyper}surfaces $\M_1$ onto $\M_2$ via a narrow band approach.
\item A volume energy $E_{\text{vol}}$ enforces a regular deformation on the whole computational domain $\Omega$.  
\end{itemize}
Our approach is based on level sets. Hence, we replace the integration over a single \added{hyper}surface, \ie $\M_1$, for the first three energies by a weighted integration over a narrow band of width $\sigma$ with 
$0 < \sigma < \Dist(\M_1, \Sing \dist_1)$. To this end we will make use of a cutoff function $\eta_\sigma \in C^\infty_0(\R)$ with $\int_\R \eta_\sigma(t) \dd t =1$ and 
$\supp\, \eta_\sigma = [-\sigma, \sigma]$. Additionally, $\eta_\sigma$ is assumed to be even and strictly decreasing in $[0,+\infty)$.

In what follows we introduce the four energy contributions separately. 

\subsection{Tangential distortion energy}

Picking up the insight gained in Section \ref{sec:tangential} we formulate the membrane energy in terms of the length and area change associated with the tangential distortion $\Di{tg}\phi$: 
\begin{equation}\label{def:Emem}
E_{\text{mem}}[\phi]= \delta \int_\Omega \eta_\sigma (\dist_1(x)) \W\big( \Di{tg}\phi(x) + \n_2(\phi(x)) \otimes \n_1(x) \big) \,\dd x,
\end{equation}
where $\W$ is a nonnegative polyconvex energy density vanishing at $\SO(n)$.
The weight $\delta$ reflects the proper scaling of the tangential distortion energy in case of a thin shell model with shell thickness $\delta$.  

The energy \eqref{def:Emem} vanishes only on deformations $\phi$ whose Jacobian matrix $\D \phi(x)$ maps $T\M_1^{\dist_1(x)}$ isometrically onto $T\M_2^{\dist_2(\phi(x))}$ for every point $x \in \supp{\eta_\sigma \circ \dist_1}$. In consequence, both tangential expansion and compression are penalized. 

Let us remark, that the extension $\Di{tg}\phi + \n_2(\phi(x)) \otimes \n_1(x)$ 
of the tangential derivative $\Di{tg}\phi$ defined in \eqref{def:Dtt} with rank $n-1$ can degenerate or be orientation-reversing depending on the 
local configuration of $\M_1$ and $\M_2$ at $x$ (cf. Figure \ref{fig:detDtt} for examples). 
\begin{figure}[ht]
    \begin{center}
        \includegraphics[width=.9\textwidth]{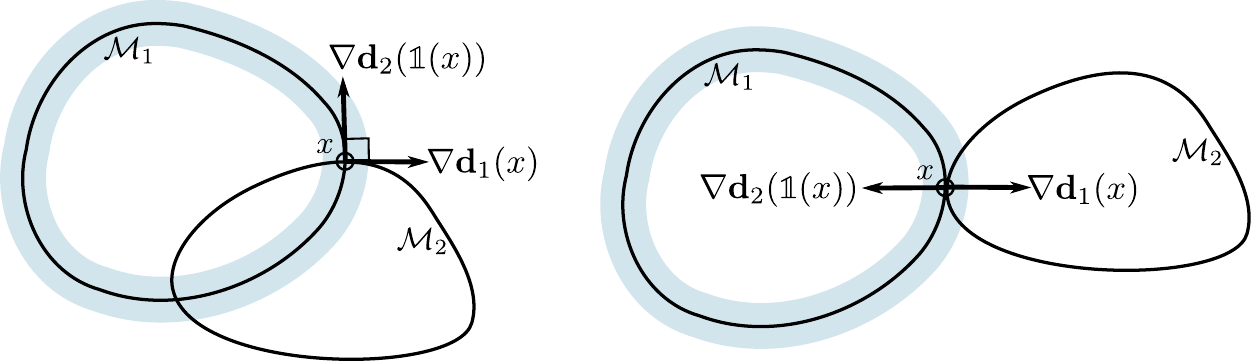}
    \caption{Configurations in which for the (obviously isometric) identity we have
    $\det \left( \Di{tg}\Id(x) + \n_2(\Id x) \otimes \n_1( x ) \right) = 0$ (left) and $\det \left( \Di{tg}\Id(x) + \n_2(\Id x) \otimes \n_1( x ) \right) < 0$ (right) and thus the extended tangential derivative degenerates or reverses orientation.}
    \label{fig:detDtt}
    \end{center}
\end{figure}

Furthermore, the energy density $W$ should not satisfy $\W(B) \to \infty$ for $\det B \to 0$. A straightforward modification of the arguments of Ciarlet and Geymonat (\cite{CiGe82}, \cite{Cia88} Theorem 4.10-2) leads to \added{a smooth} integrand $\W$ \added{which has isometries as local minimizers, } with the correct invariance properties, and with a Hessian for $B=\Id$ which matches the quadratic energy integrand of the Lam\'{e}-Navier model of linearized elasticity.
With given Lam\'{e} coefficients $\lambda, \mu > 0$, we select the energy 
\[\W(A)=\frac{\mu}{2}| A |^2 + \frac{\lambda}{4} (\det A)^2 + 
\left(\mu + \frac{\lambda}{2}\right)e^{- ( \det A - 1 ) } - \frac{(n+2) \mu}{2} - \frac{3\lambda}{4}\,.\]
This density fits into the notation of \eqref{eq:genericDensity}, if we choose $p=q=2$ and $\Gamma(t)=c t^2 + d e^{- (t-1) }\;.$
\subsection{Bending energy}

Now, we discuss a variational formulation of the curvature matching condition $\D \phi(x)^T \mathcal{S}_2(\phi(x)) \D \phi(x) = \mathcal{S}_1(x)$, which is equivalent to a vanishing relative shape operator  (cf. \eqref{eq:extshaperel}), 
where  $\mathcal{S}_i = \D \n_i \P_i = \P_i \D^2 \dist_i \P_i$ for $i=1,2$ are the shape operators on the \added{hyper}surfaces $\M_1$ and $\M_2$, respectively.
At first sight, it appears natural to formulate a quadratic penalization  and to define a bending energy 
$$
\tilde E_{\text{bend}}[\phi]= \delta^3  \int_\Omega \eta_\sigma (\dist_1(x)) |\D\phi(x)^T \mathcal{S}_2(\phi(x)) \D\phi(x) - \mathcal{S}_1(x)|^2 \dd x\,.
$$
The weight $\delta^3$ reflects the scaling of the bending energy for thin shells of 
thickness $\delta$.
However, this energy is in general not weakly lower semicontinuous. \added{Indeed, consider a situation in which $\mathcal{S}_1(x) = \mathcal{S}_2(\phi(x)) = P(e_n)$. A short computation shows that the corresponding density is not convex in its matrix variable along the rank-one segment joining $\Id-\frac{3}{4}e_1 \otimes e_1$ and $\Id-\frac{1}{2}e_1 \otimes e_1$, which precludes lower semicontinuity (cf. also 
Example \ref{ex:norank1conv} below on the lack of rank-one convexity). In fact, this kind of density is closely related to the Saint Venant-Kirchhoff energy, whose quasiconvex envelope is computed in \cite{LedRao95a}.} 

Thus, we are asking for an alternative lower semicontinuous energy functional which gives preference to 
deformations $\phi$ for which $\D\phi^T ( \mathcal{S}_2\circ \phi ) \D\phi$ is close to $\mathcal{S}_1$. 
We show that this can be achieved with the extended shape operators 
$\mathcal{S}^{ext}_i = \P_i \D^2 \dist_i \P_i + \n_i \otimes \n_i$ for $i=1,2$ and factorization. 
For proving this we make use of the following lemma.
\begin{lemma}[modified curvature matching condition]\label{lemma:MN}
Assume that $M$, $N$ are two symmetric, positive definite matrices satisfying 
\begin{equation*}
M = \P_1 M \P_1 + \n_1 \otimes \n_1 \text{ and } N = \P_2 N \P_2 + \n_2\otimes \n_2\;.
\end{equation*}
Moreover, assume that $A\in \R^{n \times n}$ satisfies 
\begin{equation*}
 A \P_1 = \P_2 A\,,
\end{equation*}
then the following statements are equivalent:
\begin{equation}\label{eq:a}
A^T \P_2 N \P_2 A =  \P_1 M \P_1
\end{equation}
and 
\begin{equation}\label{eq:b}
\Lambda[M,N,A] \added{:}= \P_2 N^{\frac12} \P_2 A \P_1 M^{-\frac12} \P_1 + \n_2 \otimes \n_1 \in \O(n)\,.
\end{equation}
\end{lemma}
\begin{proof}
By definition, the matrix $\Lambda[M,N,A]$ is orthogonal if $\Lambda[M,N,A]^T \Lambda[M,N,A] = \Id$. Therefore, if \eqref{eq:b} holds, then 
\begin{eqnarray*}
\Id &=& (\P_2 N^{\frac12} \P_2 A \P_1 M^{-\frac12} \P_1 + \n_2 \n_1^T)^T (\P_2 N^{\frac12} \P_2 A \P_1 M^{-\frac12} \P_1 + \n_2  \n_1^T) \\
&=& \P_1 M^{-\frac12} \P_1 A^T \P_2 N^{\frac12} \P_2 \P_2 N^{\frac12} \P_2 A \P_1 M^{-\frac12} \P_1 + \n_1  \n_2^T \n_2  \n_1^T \\
&=& \P_1 M^{-\frac12}  A^T \P_2 (\P_2 N^{\frac12} \P_2)^2 \P_2 A  M^{-\frac12} \P_1 + \n_1  \n_1^T \,.
\end{eqnarray*}
If we multiply this equation from left and right by $\P_1 M^{\frac12} \P_1$ and 
take into account that $(\P_2 N^{\frac12} \P_2)^2 = \P_2 N \P_2$ and $(\P_1 M^{\frac12} \P_1)^2 = \P_1 M \P_1$ we see that this is equivalent to 
$$
\P_1 A^T \P_2 N \P_2  A \P_1 = \P_1 M \P_1\,.
$$ 
Applying that $A \P_1 = \P_2 A$ we finally achieve at the equivalent condition 
$$
A^T \P_2 N \P_2  A = \P_1 M \P_1\,.
$$
The proof of the converse follows the same steps in opposite direction.
\end{proof}
If the assumptions of this lemma apply to $M=\mathcal{S}^{ext}_1(x)$, $N=\mathcal{S}^{ext}_2(y)$, and 
$A= \D \phi(x)$ with $y=\phi(x)$, then the curvature matching condition 
\begin{equation}\label{eq:matchingObjective}(\D \phi(x))^T \P_2(y) \mathcal{S}^{ext}_2(y) \P_2(y) \D \phi(x) = \P_1(x) \mathcal{S}^{ext}_1(x) \P_1(x)\end{equation}
is equivalent to  $\Lambda(\mathcal{S}^{ext}_1(x),\mathcal{S}^{ext}_2(\phi(x)) ,\D \phi(x)) \in \O(n)$
and a lower semicontinuous energy functional penalizing deviations of 
$\Lambda(\mathcal{S}^{ext}_1(x),\mathcal{S}^{ext}_2(\phi(x)) ,\D \phi(x))$ from $\O(n)$ would be a proper choice 
for realizing curvature matching. 
Unfortunately, the positive definiteness assumption of Lemma \ref{lemma:MN} is not fulfilled if principal curvatures 
of $\M_1$ or $\M_2$ are negative. 
Hence, we are replacing the extended shape operator matrices $\mathcal{S}_i^{ext}$ by symmetric and positive definite curvature classification matrices $C_i = \mathcal{C}(\mathcal{S}_i^{ext})$, $i=1,2$, respectively.

We have experimented with two different choices for $\mathcal{C}$:
\begin{itemize}
\item A simple choice is $\mathcal{C}(\mathcal{S}_i^{ext})= \mathcal{S}_i^{ext} + \mu \Id$, where $-\mu$ is a strict lower bound of the principal curvatures.
But in applications surfaces are frequently characterized by strong creases or rather sharp edges, leading to very large $\mu$. As a consequence the relative difference of the eigenvalues 
is significantly reduced when dealing with the resulting curvature classification matrices. Thus, the variational approach is less sensitive to different principal curvatures of the input \added{hyper}surfaces.
\item Another option is to use a truncation of the absolute value function for the eigenvalues of symmetric matrices. For a symmetric matrix $B\in \R^{n,n}$ with eigenvalues 
$\lambda_1,\ldots, \lambda_n$ and a diagonalization $B=Q^T \mathrm{diag}(\lambda_1,\ldots, \lambda_n) Q$ we use the classification operator 
$$\mathcal{C}(B) = Q^T \mathrm{diag}(|\lambda_1|_\tau,\ldots, |\lambda_n|_\tau) Q\,,$$
where $|\lambda|_\tau = \max\{|\lambda|, \tau\}$ for some $\tau >0$.  This approach properly represents the exact shape operator matching objective in case of principal curvatures of equal sign and absolute value larger than $\tau$. A disadvantage of this construction is that it is not able to force the deformation to correctly match curvature directions on the \added{hyper}surface with the same absolute value of the principal curvatures but with different signs. That is, locally a saddle point of the \added{hyper}surface may be mistaken for an elliptical point. However, this effect is usually compensated globally, and in applications the ansatz performs well, in particular in matching regions of edges and creases (see Section \ref{sec:appl}).
\end{itemize}
Like for the membrane energy \eqref{def:Emem}, if $\D \phi(x)$ is ensured to be orientation-preserving ($\det \D \phi >0$) and $\n_1 \cdot ( \n_2 \circ \phi ) > 0$ (cf. Figure \ref{fig:detDtt}), the curvature matching condition is equivalent to $$\Lambda(\mathcal{C}(\mathcal{S}^{ext}_1(x)),\mathcal{C}(\mathcal{S}^{ext}_2(\phi(x))) ,\D \phi(x)) \in \SO(n).$$

Based on these considerations, a suitable choice for the bending energy is
\begin{equation}\label{def:Ebend}
E_{\text{bend}}[\phi]=
\delta^3 \int_\Omega \eta_\sigma (\dist_1(x)) \W \big( \Lambda\left(\mathcal{C}(\mathcal{S}^{ext}_1(x)),\mathcal{C}(\mathcal{S}^{ext}_2(\phi(x))) ,\D \phi(x)\right) \big) \,\dd x,
\end{equation}
where $\W$ can be chosen as the same polyconvex density already used for $E_{\text{mem}}$. 
\subsection{Mismatch penalty and volumetric regularization energies}
So far, we have defined tangential membrane and bending energies which quantify the appropriateness of 
deformations $\phi:\Omega \to \R^n$ in a narrow band around the \added{hyper}surface $\M_1$.
In the derivation of these energies we assumed the constraint $\phi(\M_1)=\M_2$. 
However, such a constraint would be very hard to enforce numerically. Thus we use a weaker mismatch penalty instead:
\begin{equation}\label{def:Ematch}
E_{\text{match}}[\phi]=\frac{1}{\nu}\int_\Omega ( \eta_\sigma \circ \dist_1 ) \big| \dist_2 \circ \phi - \dist_1 \big|^2\, \dd x\,,
\end{equation}
where $1/\nu$ is a penalization parameter.

Moreover, we aim for a regular deformation on the whole computational domain $\Omega$ which is globally injective.
This, in particular, prevents from self-intersections of the deformed \added{hyper}surface $\phi(\M_1)$. 
To achieve this we introduce the following volume regularization term based on a polyconvex density $\hat{W}$ 
that enforces orientation preservation
\begin{equation}\label{def:Evol}
E_{\text{vol}}[\phi] =\begin{cases}\int_\Omega \hat{W}(\D\phi,\cof \D\phi,\det \D\phi)\,\dd x &\text{if } \det \D \phi(x) > 0 \text{ for a.e. }x, \\
+\infty &\text{otherwise.}
\end{cases}\; \text{, where }
\end{equation}
\begin{equation}\label{def:What}
{\hat  W}(\D\phi,\cof \D\phi,\det \D\phi) = \alpha_p |\D\phi|^{p}+ \beta_q |\cof \D\phi|^{q} + \gamma_s (\det \D\phi)^{-s},
\end{equation}
with $p>n$, $q>n$, $s>(n-1)q/(q-n)$, and with $\alpha_p, \beta_q, \gamma_s >0$ ensuring that the density $\hat{W}$ 
attains \added{a local minimum when } $\D\phi^T \D\phi = \Id$. As mentioned in the introduction, such an energy is weakly lower semicontinuous in $W^{1,p}(\Omega; \R^n)$ when restricted to deformations whose Jacobian determinant is positive almost everywhere, and this condition is closed under weak convergence. 

\added{Choosing $p=q=n+1$, $s=n^2$ and using that because of its symmetries $\hat{W}$ can be expressed in terms of singular values \cite[Proposition 5.31]{Dac08}, elementary but lengthy computations yield the stationarity condition at $\Id$
$$(n+1)\, n^{\frac{n-1}{2}}\big(\alpha_p + (n-1)\beta_q\big)=n^2 \gamma_s,$$
and that the corresponding Hessian is positive definite. For $n=3$ an adequate example is then } $p=q=4$, $s=9$, $\alpha_p=1, \beta_q=1$, and $\gamma_s=4$. For $n=2$, \added{one can use $p=q=3$, $s=4$, $\alpha_p=\beta_q=2$} and $\gamma_s=3 \sqrt{2}$. \added{Notice that in this case, $|\D\phi|=|\cof \D\phi|$.}

\subsection{Total energy}
Summing the above terms, our energy for shape-aware level set matching reads
\begin{equation}\label{def:energy}
E_\nu[\phi]=E_{\text{match}}[\phi]+E_{\text{mem}}[\phi]+E_{\text{bend}}[\phi]+E_{\text{vol}}[\phi],
\end{equation}
where the different terms depend on the fixed input geometries $\M_1$ and $\M_2$ through $\dist_1$ and $\dist_2$. 

\section{Existence of optimal matching deformations}
First we prove the following weak continuity lemma, which is a generalization of the classical result 
given in \cite[Theorem 4.1]{Mur87}. Here the coefficients may depend on the deformed configuration.
\begin{lemma}
\label{lem:cont} Let $\phi^k \rightharpoonup \phi \in W^{1,p}(\Omega; \R^n)$ and $p>n$. 
Moreover, let $V_i \in C^0(\overline{\Omega} \times \R^{n}; \S^{n-1})$, $i=1,2$ and we denote
\begin{equation*}
 \added{\v}_i^k(\cdot):=V_i\big(\cdot, \phi^k(\cdot)\big) \text{ and } \added{\v}_i:=V_i\big(\cdot, \phi(\cdot)\big)\,,\quad i=1,2\;.
\end{equation*}
Then 
\begin{equation}\label{eq:first}
 \det\big( \P(\added{\v}_2^k) \D \phi^k \P(\added{\v}_1^k) + \added{\v}_2^k \otimes \added{\v}_1^k \big ) \xrightharpoonup{L^{\frac{p}{n}}} \det\big( \P(\added{\v}_2) \D\phi\, \P(\added{\v}_1) + \added{\v}_2 \otimes \added{\v}_1 \big )\,.
\end{equation}
Moreover, for every \added{symmetric positive definite} $M_i$, $i=1,2$ with $M_1^{-\frac{1}{2}} \in C^0(\overline{\Omega} \times \R^{n}; \R^{n \times n})$ and $M_2^{\frac{1}{2}} \in C^0(\overline{\Omega} \times \R^{n}; \R^{n \times n})$ and the corresponding compositions 
\begin{equation*}
 M_i^k(\cdot):=M_i\big(\cdot, \phi^k(\cdot)\big) \text{ and } \overline{M}_i:=M_i\big(\cdot, \phi(\cdot)\big)
\end{equation*}
we have
\begin{equation}\label{eq:second}
\begin{aligned}
~ & \det \Lambda(M_1^k,M_2^k,\D\phi_k) = 
 \det\big( \P(\added{\v}_2^k) (M_2^k)^{\frac{1}{2}} \P(\added{\v}_2^k) \D \phi^k \P(\added{\v}_1^k) (M_1^k)^{-\frac{1}{2}} \P(\added{\v}_1^k) + \added{\v}_2^k \otimes \added{\v}_1^k \big )\\
\xrightharpoonup{L^{\frac{p}{n}}} & 
\det\big( \P(\added{\v}_2) (\overline{M}_2)^{\frac{1}{2}} \P(\added{\v}_2) \D\phi\, \P(\added{\v}_1) (\overline{M}_1)^{-\frac{1}{2}} \P(\added{\v}_1) + \added{\v}_2 \otimes \added{\v}_1 \big ) = 
\det \Lambda(\overline{M}_1,\overline{M}_2,\D\phi)\;.
\end{aligned}
\end{equation}
\end{lemma}
\begin{proof} 
To prove \eqref{eq:first} let $\zeta \in L^{\frac{p}{p-n}}(\Omega)$. We show that  
\[I^k:=\!\!\int_\Omega \zeta \det\big( \P(\added{\v}_2^k) \D \phi^k \P(\added{\v}_1^k) + \added{\v}_2^k \otimes \added{\v}_1^k \big )\,\dd x \; \rightarrow \; I:=\!\!\int_\Omega \zeta \det\big( \P(\added{\v}_2) \D \phi\, \P(\added{\v}_1) + \added{\v}_2 \otimes \added{\v}_1 \big )\,\dd x.\]
Moreover, we denote 
\begin{equation*}
 \overline{I}^k:=\int_\Omega \zeta \det\big( \P(\added{\v}_2) \D \phi^k \P(\added{\v}_1) + \added{\v}_2 \otimes \added{\v}_1 \big )\,\dd x\;.
\end{equation*}
Using the inequality (cf. \cite[Theorem 4.7]{Fri82})
$$|\det A - \det B | \leq C |A-B|\max(|A|, |B|)^{n-1}$$ 
and H\"{o}lder's inequality it follows that 
\begin{eqnarray*}
\left| I^k - \overline{I}^k \right| &\leq&
C \int_\Omega |\zeta| \left| \P(\added{\v}_2^k) \D \phi^k \P(\added{\v}_1^k) - \P(\added{\v}_2) \D \phi^k \P(\added{\v}_1) + \added{\v}_2^k \otimes \added{\v}_1^k - \added{\v}_2 \otimes \added{\v}_1 \right|  \\
&&\enskip \cdot \max\bigg(\big|\P(\added{\v}_2^k) \D \phi^k \P(\added{\v}_1^k) + \added{\v}_2^k \otimes \added{\v}_1^k\big|, \big|\P(\added{\v}_2) \D \phi^k \P(\added{\v}_1) + \added{\v}_2 \otimes \added{\v}_1 \big|\bigg)^{n-1} \dd x\\
&\leq& C \,\|\zeta\|_{L^\frac{p}{p-n}} \left\| \big|\D \phi^k\big|^{n-1} + 1 \right\|_{L^{\frac{p}{n-1}}} \\
&&\enskip \cdot \left\| \P(\added{\v}_2^k) \D \phi^k \P(\added{\v}_1^k) - \P(\added{\v}_2) \D \phi^k \P(\added{\v}_1) + \added{\v}_2^k \otimes \added{\v}_1^k - \added{\v}_2 \otimes \added{\v}_1 \right\|_{L^p}\\
&\leq& C \, \|\zeta\|_{L^\frac{p}{p-n}} \Big( \|\D \phi^k\|_{L^p}^{n-1} + 1 \Big) \\ 
&&\enskip \cdot \Big[ \|\D \phi^k\|_{L^p} \left( \|\P(\added{\v}_2^k)\|_{L^\infty}\|\P(\added{\v}_1^k)-\P(\added{\v}_1)\|_{L^\infty} + \|\P(\added{\v}_1)\|_{L^\infty}\|\P(\added{\v}_2^k)-\P(\added{\v}_2)\|_{L^\infty}\right) \\
&&\quad+ \Big( \|\added{\v}_1^k- \added{\v}_1\|_{L^\infty} + \|\added{\v}_2^k- \added{\v}_2\|_{L^\infty} \Big) \Big]\,.
\end{eqnarray*}
Here, we have used that 
$$\big|\P(\added{\v}_2) \D \phi^k \P(\added{\v}_1) + \added{\v}_2 \otimes \added{\v}_1 \big|^{n-1} \leq ( \big|\D \phi^k\big| + 1 )^{n-1} \leq C ( \big|\D \phi^k\big|^{n-1} + 1 )\,.$$
By the Rellich-Kondrakov embedding theorem (\cite{Ada03}, Theorem 6.3 III)  there exist
subsequences of $\added{\v}_i^k$, $i=1,2$, which for simplicity of notation are again denoted by $\added{\v}_i^k$, $i=1,2$, 
that converge uniformly to $\added{\v}_i$, $i=1,2$, respectively. 
Taking into account the Lipschitz continuity estimate $$|\P(e)-\P(f)|=|(e-f)\otimes e + f \otimes (e-f)| \leq 2 \sqrt{n} |e - f|$$ 
and that $\added{\v}_i^k \to \added{\v}_i$, $i=1,2$ in $L^\infty$ we obtain $| I^k - \overline{I}^k | \rightarrow 0$ for $k\to \infty$.
 
Next, we replace $\added{\v}_i$, $i=1,2$ in $\bar I^k$ by a piecewise constant approximation 
on a grid superimposed to the computational domain $\Omega$.
Explicitly, we consider the finitely many non empty intersection $\omega_\delta^z = \delta (z+ [0,1]^n) \cap \Omega$ of cubical cells with 
$\Omega$ for $z \in \Z^n$ and define
$$
\bar I^k_\delta := \sum_{z\in \Z^n} 
\int_{\omega_\delta^z} \zeta \det\big( \P(\added{\v}_2(z_\delta)) \D \phi^k \P(\added{\v}_1(z_\delta)) + \added{\v}_2(z_\delta) \otimes \added{\v}_1(z_\delta) \big )\dd x\,,
$$
where $z_\delta$ is any point in $\bar \Omega \cap \omega_\delta^z$ if this set is nonempty.
Using analogous estimates as above we obtain 
\begin{eqnarray*}
\left|\overline{I}_\delta^k - \overline{I}^k \right| &\leq&
C \|\zeta\|_{L^\frac{p}{p-n}} \Big( \|\D \phi^k\|_{L^p}^{n-1} + 1 \Big)\\ 
&&\enskip\cdot \Big[ \|\D \phi^k\|_{L^p} \left( \|\P(\added{\v}_{1,\delta})\|_{L^\infty}\|\P(\added{\v}_{2,\delta})-\P(\added{\v}_2)\|_{L^\infty} + \|\P(\added{\v}_1)\|_{L^\infty}\|\P(\added{\v}_{1,\delta})-\P(\added{\v}_1)\|_{L^\infty}\right) \\
&&\quad + \Big( \|\added{\v}_{2,\delta}- \added{\v}_2\|_{L^\infty} + \|\added{\v}_{1,\delta}- \added{\v}_1\|_{L^\infty} \Big) \Big]\,,
\end{eqnarray*}
where $\added{\v}_{1,\delta}$ and $\added{\v}_{2,\delta}$ are piecewise constant functions in $L^\infty$ with $\added{\v}_{1,\delta}|_{\omega_\delta^z} = \added{\v}_1(z_\delta)$ and $\added{\v}_{2,\delta}|_{\omega_\delta^z} = \added{\v}_2(z_\delta)$, respectively. 

Using the uniform continuity of $\added{\v}_2$ and $\added{\v}_1$ on $\overline{\Omega}$ we obtain that 
$\left| \overline{I}^k_\delta - \overline{I}^k \right| \leq \beta(\delta)$ for a monotonically increasing continuous function 
$\beta: \R^+_0 \to \R$ with $\beta(0) =0$. In particular the convergence is uniform with respect to $k$.
The same argument applies for the difference of $I$ and 
\begin{equation*}
 \bar I_\delta := \sum_{z\in \Z^n} \int_{\omega_\delta^z} \zeta \det\big( \P(\added{\v}_2(z_\delta)) \D \phi \P(\added{\v}_1(z_\delta)) + \added{\v}_2(z_\delta) \otimes \added{\v}_1(z_\delta) \big )\dd x
\end{equation*}
and we get $\left| \bar I_\delta - I \right|<C\beta(\delta)$.
Using \eqref{eq:complement} it follows that 
\begin{equation*}
Q(\added{\v}_{\added{2}}(z_\delta))^T \added{\Big(}\P(\added{\v}_{\added{2}}(z_\delta)) A \P(\added{\v}_{\added{1}}(z_\delta)) + \added{\v}_{\added{2}}(z_\delta) \otimes \added{\v}_{\added{1}}(z_\delta)\added{\Big)}Q(\added{\v}_{\added{1}}(z_\delta)) = 
\left(
\begin{array}{c|c}
\tilde{A} &  0 \\ \hline
0  &  1
\end{array}
\right)\;.
\end{equation*}
Thus 
$ \det ( \P(\added{\v}_2(z_\delta)) A \P(\added{\v}_1(z_\delta)) + \added{\v}_2(z_\delta) \otimes \added{\v}_1(z_\delta)) = \det (\tilde{A})$ 
represents an $(n-1)\times (n-1)$ minor of the linear mapping corresponding to the matrix $A$ with respect to different orthogonal basis in preimage space (associated with $\P(\added{\v}_1(z_\delta))$ and $\added{\v}_1(z_\delta)$)
and the image space (associated with $\P(\added{\v}_2(z_\delta))$ and $\added{\v}_2(z_\delta)$). 
Indeed, denoting $Q_i:=Q(\added{\v}_i(z_\delta))$ we have
\begin{eqnarray*}
&&\int_{\omega_\delta^z} \zeta(x) \det\big( \P(\added{\v}_2(z_\delta)) \D \phi^k(x) \P(\added{\v}_1(z_\delta)) + \added{\v}_2(z_\delta) \otimes \added{\v}_1(z_\delta) \big )\dd x\,\\
&&=\int_{\omega_\delta^z} \zeta(x) \det\big( Q_2^T \big(\P(\added{\v}_2(z_\delta)) \D \phi^k(x) \P(\added{\v}_1(z_\delta)) + \added{\v}_2(z_\delta) \otimes \added{\v}_1(z_\delta) \big) Q_1 \big)\dd x\,\\
&&=\int_{\omega_\delta^z} \zeta(x) \det\big( Q_2^T \P(\added{\v}_2(z_\delta)) Q_2 Q_2^T \D \phi^k(x) Q_1 Q_1^T \P(\added{\v}_1(z_\delta)) Q_1 + e_n \otimes e_n \big )\dd x\,\\
&&=\int_{\omega_\delta^z} \zeta(x) \det\big( \P(e_n) Q_2^T  \D \phi^k(x) Q_1 \P(e_n) + e_n \otimes e_n \big)\dd x\,\\
&&=\int_{Q_1^T\omega_\delta^z} \zeta(Q_1y) \det\big( \P(e_n) \D \big( Q_2^T \circ \phi^k \circ Q_1\big) (y) \P(e_n) + e_n \otimes e_n \big)\dd y\,\\
&&=\int_{Q_1^T\omega_\delta^z} \zeta(Q_1y) \Cof_{nn}\big( \D \big( Q_2^T \circ \phi^k \circ Q_1\big) (y) \big)\dd y\,,
\end{eqnarray*}
where we have used the orthogonal change of variables $y = Q_1^T x$ and $\Cof_{nn}$ denotes the minor obtained by erasing the last column and the last row. This change of orthogonal coordinates is fixed on each cell $\omega_\delta^z$.
Since for each $\delta$ the domain $\Omega$ is covered by finitely many cells $\omega_\delta^z$, 
using the above computation and standard weak continuity results \cite[Theorem 8.20]{Dac08} for determinants of minors of the Jacobian we obtain that $\bar I^k_\delta \rightarrow \bar I_\delta$ for 
$k\to \infty$.
Finally, for given $\epsilon$ we first choose $\delta$ small enough to ensure that $\left| \bar I_\delta - I \right| + \left| \bar I^k_\delta - \bar I^k \right|\leq \tfrac{\epsilon}{2}$. Then we choose $k$ large enough
to ensure that $\left| I^k - \bar I^k \right| + \left| \bar I^k_\delta - \bar I_\delta \right|\leq \tfrac{\epsilon}{2}$. This proves that a subsequence of $I^k$ converges to $I$ for $k\to \infty$. 
Since the limit does not depend on the subsequence, we finally obtain weak convergence for the whole sequence.

To prove \eqref{eq:second}, consider the three sequences of matrix functions
\begin{equation}\label{eq:factors}
\P(\added{\v}_2^k) (M_2^k)^{\frac{1}{2}} \P(\added{\v}_2^k) + \added{\v}_2^k \otimes \added{\v}_2^k,\; 
\P(\added{\v}_2^k) \D \phi^k \P(\added{\v}_1^k) + \added{\v}_2^k \otimes \added{\v}_1^k \text{ and }
\P(\added{\v}_1^k) (M_1^k)^{-\frac{1}{2}} \P(\added{\v}_1^k) + \added{\v}_1^k \otimes \added{\v}_1^k.
\end{equation}
The determinant of the second expression above converges weakly as $k \to \infty$ by the first part of the lemma, while the determinants of the first and third can be assumed to converge uniformly.
Moreover, the matrices in \eqref{eq:factors} have the block structure shown in \eqref{eq:complement}, so multiplying the three together and taking into account that $\P$ is a projection (depending on the argument) recovers the matrix
\[\P(\added{\v}_2^k) (M_2^k)^{\frac{1}{2}} \P(\added{\v}_2^k) \D \phi^k \P(\added{\v}_1^k) (M_1^k)^{-\frac{1}{2}} \P(\added{\v}_1^k) + \added{\v}_2^k \otimes \added{\v}_1^k\]
appearing in the statement. Multiplicativity of the determinant and the fact that a product of strongly converging and one weakly converging sequence converges weakly then finishes the proof.
\end{proof}


We are now in a position to prove existence of a minimizing deformation for the \added{hyper}surface matching energy $E$ in a suitable set of admissible deformations.
Of particular difficulty is that derivatives of $\dist_2$ are not defined in the whole of $\Omega$ and that in the functional these derivatives are evaluated at deformed positions.
We handle this by ensuring that the involved deformations are such that terms involving these derivatives are not evaluated near the singularities.
We obtain the following theorem:

\begin{theorem}[Existence of minimizing deformations]\label{thm:existence} Let $\M_1,\M_2$ be $C^{2,1}$ compact embedded \added{hyper}surfaces in $\R^n$ such that a $C^1$ diffeomorphism $\varphi: \M_1 \to \M_2$ exists between them. 

Assume further that
\begin{equation}\label{eq:sigmaCond}0 < \sigma < \min ( \Dist( \M_1, \Sing \dist_1 ), \Dist( \M_2, \Sing \dist_2 ) ),\end{equation}
where $\Sing\dist_i$ is the set of points where $\dist_i$ is not differentiable, and that $\added{\mathcal{C}}:\R^{n \times n} \to \SPD(n)$ is continuous. Then there exists a constant
$0 < \nu_0 := \nu_0(\Omega, \M_1, \M_2, \sigma, p, \alpha_p)$
such that for $0 < \nu \leq \nu_0$, the functional $E_\nu$ has at least one minimizer $\phi$ among deformations in the space $W^{1,p}_0(\Omega; \R^{n})+\id$. Moreover, $\phi$ is a homeomorphism of $\Omega$ into $\Omega$, and $\phi^{-1}\in W^{1,\theta}(\Omega; \R^{n})$, where $\theta$ is given by $\theta=q(1+s)/(q+s)$.
\end{theorem}

\begin{proof} We proceed in several steps.

\noindent \emph{Step 1: Coercivity.}  First we point out the coercivity enjoyed by our functional. Using the Poincar\'{e} and Morrey inequalities (\cite{Leo09}, Theorem 12.30 and 11.34), and the Dirichlet boundary conditions we have
\begin{equation}\label{eq:coercivity}\|\phi\|_{C^{0,\alpha}(\Omega)} \leq C \|\phi\|_{W^{1,p}(\Omega)} \leq C ( 1 + \|\D\phi\|_{L^p(\Omega)}) \leq C (1 + E_\nu[\phi]^{\frac{1}{p}}),\end{equation}
for any $\phi \in W^{1,p}_0(\Omega)+\id$ and $\alpha=1-n/p$.
\\
\noindent \emph{Step 2: Lower semicontinuity along sequences of constrained deformations.}
For the remainder of the proof, a deformation $\phi \in W^{1,p}_0(\Omega; \R^n) + \id$, $p>n$ is 
termed $\rho$-admissible for $\rho > 0 $, if 
\begin{itemize}
 \item $E_{\text{vol}}[\phi]< +\infty$, 
 \item $\det \D \phi(x) > 0$ for a.e. $x \in \Omega$, and 
 \item for all $x \in \supp\, \big( \eta_\sigma \circ \dist_1 \big)$ 
       and every $y \in \Sing(\dist_2)$, we have  $\left| \phi(x) - y \right|\geq \rho$.
\end{itemize}
Notice that since $p>n$, $\phi$ has a unique continuous representative, so the third property is well defined.

First, notice that with the assumption \eqref{eq:sigmaCond} we have 
\begin{equation}
\label{eq:dist}
 \supp (\eta_\sigma \circ \dist_i) = \{|\dist_i|\leq \sigma\} \subset \Omega \setminus \overline{\Sing(\dist_i)}\,,\quad i=1,2.
\end{equation}
Let $\phi^k$ be a sequence of $\rho$-admissible deformations with $E_{\text{vol}}[\phi^k]\leq C$. 
By \eqref{eq:coercivity} and using the Banach-Alaoglu and Rellich-Kondrakov theorems, 
a subsequence (again denoted by $(\phi^k)$) converges to a deformation $\phi$, both in the $W^{1,p}$-weak and 
uniform topologies. 

Now, we have (\cite[Theorem 8.20]{Dac08})
\begin{equation}\label{eq:weakConvDet}
(\det \D \phi^k, \cof \D \phi^k) \rightharpoonup (\det \D \phi, \cof \D \phi ) 
\text{ in } L^{\frac{p}{n}}(\Omega) \times \left( L^{\frac{p}{n-1}}(\Omega)\right)^{n^2}.
\end{equation}

Additionally, since \eqref{eq:weakConvDet} holds
and because $E_\nu[\phi_k]$ is \added{bounded}, $\int_\Omega (\det \D \phi_k)^{-s} d x$ 
is \added{bounded} by the definition of $\hat{W}$ and 
$\det \D \phi_k \geq 0$ a.e. Together with \eqref{eq:weakConvDet}, we have
\begin{equation}
\label{eq:admis2}
\det \D \phi(x) > 0 \text{ a.e.,} 
\end{equation}
so that $\phi$ is again $\rho$-admissible.
 
Notice also that by a.e. positivity of the determinants, \eqref{eq:weakConvDet} and a standard lower 
semicontinuity result for convex integrands (see e.g. \cite[Theorem 3.23]{Dac08}) implies
\[E_{\text{vol}}[\phi] \leq \liminf_{k \to \infty} E_{\text{vol}}[\phi^k],\]
and uniform convergence of $\phi^k$ immediately leads to
\[E_{\text{match}}[\phi] = \lim_{k \to \infty} E_{\text{match}}[\phi^k].\]

We claim that under the assumptions of this theorem, we also have that
\begin{equation}
\label{eq:Emem-lsc}
E_{\text{mem}}\left[\,\phi\,\right] \leq \liminf_{k \to \infty} E_{\text{mem}}[\phi^k]
\end{equation}
and 
\begin{equation}
\label{eq:Ebend-lsc}
E_{\text{bend}}\left[\,\phi\,\right] \leq \liminf_{k \to \infty} E_{\text{bend}}[\phi^k].
\end{equation}
To see this, notice that $\phi^k, \phi$ being $\rho$-admissible ensures that the normal vectors satisfy
\[\n_1,\, \n_2 \circ \phi^k,\, \n_2 \circ \phi \in C^{0}( \{|\dist_1|\leq \sigma\}; \R^{n}).\]
Consequently, the first part of Lemma \ref{lem:cont} \added{(with $V_i=\n_i$)} implies
\begin{equation}
\label{eq:DttConv}\begin{aligned}\raisebox{2pt}{$\chi$}_{\{|\dist_1|\leq \sigma\}}&\Big( \Di{tg}\phi^k, \det\big(\Di{tg}\phi^k\added{+ (\n_2\circ \phi^k) \otimes \n_1}\big)\Big) \\
\rightharpoonup \raisebox{2pt}{$\chi$}_{\{|\dist_1|\leq \sigma\}}&\Big( \Di{tg}\phi, \det\big(\Di{tg}\phi\added{+ (\n_2\circ \phi) \otimes \n_1}\big)\Big)\text{ in } \left(L^p(\Omega) \right)^{n^2} \times L^{\frac{p}{n}}(\Omega),\end{aligned}\end{equation}
with $\raisebox{2pt}{$\chi$}_{\{|\dist_1|\leq \sigma\}}$ denoting the indicator function. 
Combining \eqref{eq:DttConv} with the polyconvexity of $W$, defining the function $E_{\text{mem}}$, both introduced in 
\eqref{def:Emem} we find the assertion \eqref{eq:Emem-lsc}.

Furthermore, by our assumptions on $\M_i$ (see section \ref{sec:deformation}), we have that 
$$\raisebox{2pt}{$\chi$}_{\{|\dist_i|\leq \sigma\}}\mathcal{S}_i = \raisebox{2pt}{$\chi$}_{\{|\dist_i|\leq \sigma\}}\D^2 \dist_i \in C^0( \overline{\Omega}; \R^{n \times n}).$$
Since $\mathcal{C}$ produces uniformly positive matrices, we have $\raisebox{2pt}{$\chi$}_{\{|\dist_1|\leq \sigma\}}(\mathcal{C}(\mathcal{S}^{ext}_1))^{-1}\in C^0(\overline{\Omega}; \R^{n \times n})$. We can then use a continuity result for square roots of nonnegative definite matrix-valued functions defined on $\Omega$ \cite[Theorem 1.1]{ChHu97} to see that
\[\raisebox{2pt}{$\chi$}_{\{|\dist_1|\leq \sigma\}}(\mathcal{C}(\mathcal{S}^{ext}_1))^{-\frac{1}{2}},\, \raisebox{2pt}{$\chi$}_{\{|\dist_1|\leq \sigma\}}( \mathcal{C}(\mathcal{S}^{ext}_2) \circ \phi^k )^{\frac{1}{2}},\, \raisebox{2pt}{$\chi$}_{\{|\dist_1|\leq \sigma\}}(\mathcal{C}(\mathcal{S}^{ext}_2) \circ \phi )^{\frac{1}{2}} \in C^0(\overline{\Omega}; \R^{n \times n}).\]
The second part of Lemma \ref{lem:cont} implies the weak convergence
\begin{equation*}
 \begin{aligned}
  ~ & \raisebox{2pt}{$\chi$}_{\{|\dist_1|\leq \sigma\}}\Big( \Lambda(\mathcal{C}(\mathcal{S}^{ext}_1(\phi^k)),\mathcal{C}(\mathcal{S}^{ext}_2(\phi^k)) ,\D \phi^k), \det\Lambda(\mathcal{C}(\mathcal{S}^{ext}_1),\mathcal{C}(\mathcal{S}^{ext}_2(\phi^k),\D \phi^k)\Big) \\
  &\xrightharpoonup{L^{\frac{p}{n}}}
  \raisebox{2pt}{$\chi$}_{\{|\dist_1|\leq \sigma\}}\Big( \Lambda(\mathcal{C}(\mathcal{S}^{ext}_1),\mathcal{C}(\mathcal{S}^{ext}_2(\phi)) ,\D \phi), \det\Lambda(\mathcal{C}(\mathcal{S}^{ext}_1),\mathcal{C}(\mathcal{S}^{ext}_2(\phi),\D \phi)\Big),
 \end{aligned}
\end{equation*}
from which \eqref{eq:Ebend-lsc} follows by using the polyconvexity of $W$.
\\
\noindent \emph{Step 3: Existence of minimizers restricted to admissible deformations.} 
Since we have already seen that the set of $\rho$-admissible deformations is weakly closed and weakly compact, 
and that every term of $E$ is weakly lower semicontinuous on this set, we just need to check that for all fixed 
$\nu > 0$, the set of $\rho$-admissible deformations, with adequate $\rho$, is not empty.

For some given $\sigma$ satisfying $\Dist( \M_2, \Sing \dist_2 ) - \sigma > 0$ let $\rho$ satisfy
\begin{equation}
\label{eq:const_rho}
 0 < \rho < \Dist( \M_2, \Sing \dist_2 ) - \sigma\;.
\end{equation}
We construct a deformation $\hat{\varphi}$, which is $\rho$-admissible and satisfies $E_\nu[\hat{\varphi}] < \infty$.
By assumption, there exists a diffeomorphism $\varphi:\M_1\to \M_2$. 
Thus, we construct an extension of this diffeomorphism to 
$\{|\dist_1|\leq \sigma\}$ along the normal directions using 
\begin{equation}
\label{eq:diff-ext}
\hat{\varphi}( x+s\n_1(x) ):=\varphi(x)+s \n_2(\varphi(x)), \text{ for }x \in \M_1, -\sigma \leq s \leq \sigma.
\end{equation}
We can then extend $\hat{\varphi}$ to the inside and outside components $\Omega_i, \Omega_o$ of $\Omega \setminus \{|\dist_1|\leq \sigma\}$ by solving the minimization problems for $E_{\vol}$
with Dirichlet boundary conditions given by \eqref{eq:diff-ext} on $\partial \Omega_i$ and $\partial \Omega_o \setminus \partial \Omega$, and by $\hat{\varphi}(x)=x$ on $\partial \Omega$. For the resulting $\hat{\varphi}$ we have
\[E_{\text{match}}[\hat{\varphi}]=0,\, E_{\text{vol}}[\hat{\varphi}]<\infty,\, E_{\text{mem}}[\hat{\varphi}]<\infty,\, E_{\text{bend}}[\hat{\varphi}]<\infty,\]
where the first two statements follow by construction, and the last two by virtue of $\varphi$ being a diffeomorphism and the choice of $\sigma$. Moreover, we note that since $\hat{\varphi}$ has finite energy and the growth conditions assumed for 
$\hat{W}$ (see \eqref{def:Evol}), the condition $\det \D \hat{\varphi}(x) > 0$ for a.e. $x$ is also satisfied \cite{Bal81}.
\\
\noindent \emph{Step 4: A priori estimate to remove the constraint.} 
Next, we show that for any $\rho$ satisfying \eqref{eq:const_rho} there exists a parameter $\nu_0 > 0$ such that for all $0 < \nu < \nu_0$ the 
constrained minimizers of $E_\nu$ subject to \eqref{eq:sigmaCond} solves the unconstrained optimization problem, 
consisting in minimizing $E_\nu$ on $W_0^{1,p}+\id$.

To this end, we verify that every $\phi$ that satisfies 
\begin{equation}\label{eq:energybound}
E_\nu[ \phi ] \leq E_\nu[ \hat{\varphi}]
\end{equation}
is $\rho$-admissible. It is immediate from \eqref{eq:energybound} that $E_{\text{vol}}(\phi) < + \infty$, and from the definition of $\hat{W}$ in \eqref{def:Evol} it follows with the same arguments as in \eqref{eq:admis2} that $\det \phi > 0$ a.e.

We prove now that for all deformations $\phi$ satisfying \eqref{eq:energybound} also satisfy
\begin{equation}
\label{eq:distance}
\|\dist_2 \circ \phi \|_{L^\infty(\{|\dist_1| \leq \sigma\})} \leq \Dist(\M_2,\Sing \dist_2) - \rho\;.
\end{equation}
This is sufficient because from \eqref{eq:distance} it follows for all $x$ satisfying $|\dist_1(x)| \leq \sigma$
by the triangle inequality that 
\begin{equation*}
\begin{aligned}
\rho &\leq \Dist(\M_2,\Sing \dist_2) - \|\dist_2 \circ \phi \|_{L^\infty(\{|\dist_1| \leq \sigma\})}\\
& = \Dist(\M_2,\Sing \dist_2) - \Dist(\phi(x),\M_2) \\
&\leq \Dist(\phi(x),\Sing \dist_2),
\end{aligned}
\end{equation*}
which is the third property of a $\rho$-admissible deformation $\phi$.
      
To prove \eqref{eq:distance} we use the triangle inequality and estimate
\begin{equation}\label{eq:geoestimate1}
\|\dist_2 \circ \phi \|_{L^\infty(\{|\dist_1| \leq \sigma\})} \leq 
\sigma + \|\dist_2 \circ \phi  - \dist_1 \|_{L^\infty(\{|\dist_1| \leq \sigma\})}.
\end{equation}
By the monotonicity of $\eta_\sigma$ and the fact that the signed distance functions $\dist_i$ are 
Lipschitz continuous with constant $1$ we have, for each $\hat{\sigma} \in (0, \sigma)$ that
\begin{eqnarray}\label{eq:geoestimate2}
&& \|\dist_2 \circ \phi  - \dist_1 \|_{L^\infty(\{|\dist_1| \leq \sigma\})} \nonumber \\
&&\leq \left(1 \!+\! \| \phi \|_{C^{0,\alpha}(\{\sigma - \hat{\sigma} \leq |\dist_1| \leq \sigma\})}\right)\hat{\sigma}^\alpha + \frac{ \|\eta_\sigma \circ \dist_1 (\dist_2 \circ \phi  - \dist_1) \|_{L^\infty( \{ |\dist_1| < \sigma -\hat{\sigma}\} )}}{\eta_\sigma(\sigma  - \hat{\sigma} )}.
\end{eqnarray}
Estimates \eqref{eq:coercivity} and \eqref{eq:energybound} imply in turn
\begin{equation}\label{eq:geoestimate3}
\begin{gathered}
\| \phi \|_{C^{0,\alpha}(\{\sigma - \hat{\sigma} \leq |\dist_1| \leq \sigma\})} \leq C \| \phi \|_{W^{1,p}(\Omega)} 
\leq C (1+E_\nu[\hat{\varphi}]^{\frac{1}{p}}).
\end{gathered}
\end{equation}
Finally, combining \eqref{eq:geoestimate1}, \eqref{eq:geoestimate2}, and \eqref{eq:geoestimate3} we obtain
\begin{eqnarray} \nonumber
&&\|\dist_2 \circ \phi \|_{L^\infty(\{|\dist_1| \leq \sigma\})}  \\
&& \leq \sigma + \left(1 + C ( 1 +  E_\nu[\hat{\varphi}]^{\frac{1}{p}} ) \right)\hat{\sigma}^\alpha + \frac{1}{\eta_\sigma(\sigma  - \hat{\sigma} )}
\|\eta_\sigma \circ \dist_1 (\dist_2 \circ \phi  - \dist_1) \|_{L^\infty( \{ |\dist_1| < \sigma -\hat{\sigma}\} )}.
\label{eq:geoestimate}
\end{eqnarray}
Now we can apply Ehrling's lemma \cite[Theorem 7.30]{ReRo04} for the embeddings
$ W^{1,p}(\Omega) \subset \subset  L^\infty(\Omega) \subset L^2(\Omega)$
to control the last term in \eqref{eq:geoestimate}. Taking into account the Poincar\'{e} inequality and Dirichlet boundary conditions, we obtain for any $\epsilon > 0$ a constant $C(\epsilon) > 0$ such that\begin{eqnarray} \nonumber 
&&\|\eta_\sigma \circ \dist_1 (\dist_2 \circ \phi  - \dist_1) \|_{L^\infty( \{ |\dist_1| < \sigma -\hat{\sigma}\} )} \leq \|\eta_\sigma \circ \dist_1 (\dist_2 \circ \phi  - \dist_1) \|_{L^\infty( \Omega )} \\
&& \leq C(\epsilon)\|\eta_\sigma \circ \dist_1 (\dist_2 \circ \phi  - \dist_1) \|_{L^2( \Omega )} \! + \epsilon\, C \! \left(\|\nabla ( \eta_\sigma \circ \dist_1 (\dist_2 \circ \phi  - \dist_1) )\|_{L^p( \Omega )}+1\right).
\label{eq:ehrlingestimate1}
\end{eqnarray}
Now, for the first term in the right hand side of \eqref{eq:ehrlingestimate1} we can estimate
\begin{equation}\label{eq:ehrlingestimate2}
\|\eta_\sigma \circ \dist_1 (\dist_2 \circ \phi  - \dist_1) \|_{L^2( \Omega )} = \nu^{\frac{1}{2}} E_{\text{match}}[\phi]^{\frac{1}{2}} \leq \nu^{\frac{1}{2}} E_\nu[\hat{\varphi}]^{\frac{1}{2}}.
\end{equation}
For the second term, denoting $\diam \Omega = \sup_{x,y \in \Omega}|x-y|$, 
\begin{eqnarray} 
&&\|\nabla ( \eta_\sigma \circ \dist_1 (\dist_2 \circ \phi  - \dist_1) )\|_{L^p( \Omega )}\nonumber  \\ 
&&\leq \|\nabla ( \eta_\sigma \circ \dist_1 ) ( \dist_2 \circ \phi  - \dist_1 ) \|_{L^p(\Omega)} + \| ( \eta_\sigma \circ \dist_1 ) \nabla ( \dist_2 \circ \phi  - \dist_1 ) \|_{L^p(\Omega)} + 1\nonumber \\
&&\leq C \nu^{\frac{1}{p}} \left( \| \dist_2 \circ \phi - \dist_1 \|^{\frac{p-2}{p}}_{L^\infty( \Omega )} E_{\text{match}}[\phi]^{\frac{1}{p}}\right) + C \big( \| \D \phi \|_{L^p(\Omega)} + \| \nabla \dist_1 \|_{L^{p}(\Omega)} + 1 \big) \nonumber\\
&&\leq C \nu^{\frac{1}{p}} \left( ( \|\phi\|_{C^{0,\alpha}(\Omega)} + 2 \diam \Omega )^{\frac{p-2}{p}}E_\nu[\hat{\varphi}]^{\frac{1}{p}} \right) + C \! \left( E_\nu[\hat{\varphi}]^{\frac{1}{p}} + 1 \right) \nonumber\\
&&\leq C \nu^{\frac{1}{p}} \left( ( 1 + E_\nu[\hat{\varphi}]^{\frac{1}{p}} )^{\frac{p-2}{p}} E_\nu[\hat{\varphi}]^{\frac{1}{p}}\right) + C\! \left( E_\nu[\hat{\varphi}]^{\frac{1}{p}} + 1\right),
\label{eq:ehrlingestimate3}
\end{eqnarray}
where we have applied the product rule, the definition of $E_{\text{match}}$, 
$\eta_\sigma \in C^\infty_0$, $\eta_\sigma \leq C$, that $| \nabla \dist_i | = 1$ a.e., $i=1,2$, the chain rule, and \eqref{eq:energybound}. The use of the chain rule is justified by \cite[Theorem 2.2]{MaMi72}, since $\dist_2$ has Lipschitz constant $1$. 

Together, \eqref{eq:ehrlingestimate1}, \eqref{eq:ehrlingestimate2}, and \eqref{eq:ehrlingestimate3} imply
\begin{eqnarray} \nonumber
&&\|\eta_\sigma \circ \dist_1 (\dist_2 \circ \phi  - \dist_1) \|_{L^\infty( \{ |\dist_1| < \sigma -\hat{\sigma}\} )} \\
&&\leq \nu^{\frac{1}{p}} \left( C(\epsilon) \nu^{\frac{1}{2}-\frac{1}{p}} E_\nu[\hat{\varphi}]^{\frac{1}{2}} + \epsilon\, C ( 1 + E_\nu[\hat{\varphi}]^{\frac{1}{p}} )^{\frac{p-2}{p}} E_\nu[\hat{\varphi}]^{\frac{1}{p}}\right) + \epsilon\, C\! \left( E_\nu[\hat{\varphi}]^{\frac{1}{p}} + 1\right).
\label{eq:ehrlingestimate}
\end{eqnarray}
In light of \eqref{eq:geoestimate} and \eqref{eq:ehrlingestimate}, and since $E_\nu[\hat{\varphi}]$ is independent of $\nu$, we can now choose first $\hat{\sigma}$, then $\epsilon$ and finally $\nu$ small enough to obtain
\[ \| \dist_2 \circ \phi \|_{L^\infty(\{|\dist_1| \leq \sigma\})} \leq \sigma + ( \Dist( \M_2, \Sing \dist_2 ) - \sigma - \rho) \leq \Dist( \M_2, \Sing \dist_2 ) - \rho\;.\]
\noindent \emph{Step 5: Injectivity.}
The injectivity and regularity of the inverse follow by the growth conditions satisfied by $E_\text{vol}$ and classical results of Ball \cite[Theorems 2 and 3]{Bal81}. \added{Note that Theorem 3 in \cite{Bal81} is stated in the mechanical application context in dimension $n=3$, but it holds also in $\R^n$ following the same proof and using the condition $s>(n-1)q/(q-n)$.}
\end{proof}

We have particularized the statement of Theorem \ref{thm:existence} to the case of  Dirichlet boundary conditions to ensure global invertibility. In fact, we also have existence of minimizing deformations for the case of Neumann boundary conditions.

\begin{corollary}[Natural boundary conditions]\label{cor:neumann}Under the assumptions of Theorem \ref{thm:existence} above, there exists a constant 
\[ 0 < \nu_N = \nu_N(\Omega, \M_1, \M_2, \sigma, p, \alpha_p)\]
such that for $0 < \nu \leq \nu_N$, the functional $E_\nu$ possesses at least one minimizer among deformations in the space $W^{1,p}(\Omega; \R^{n})$.
\end{corollary}

\begin{proof}
The proof follows the same arguments used for Theorem \ref{thm:existence}, so we only point out the necessary modifications. 
We need a replacement for the coercivity estimate \eqref{eq:coercivity} and claim 
\begin{equation}\label{eq:modifiedEstimate}\|\phi\|_{W^{1,p}(\Omega)} \leq  
C(1 + \nu^{\frac{1}{2}} E_{\text{match}}[\phi]^{\frac{1}{2}} + \|\D\phi\|_{L^p(\Omega)}) 
\leq C (1 + \nu^{\frac{1}{2}} E_\nu[\phi]^{\frac{1}{2}}+ E_\nu[\phi]^{\frac{1}{p}} ).\end{equation}
To verify this let us consider $\omega:=\{|\dist_1| \leq \sigma/2 \}$.
An adequate Poincar\'{e} inequality (see e.g. \cite[Theorem 12.23]{Leo09}) implies that
\[\|\phi\|_{W^{1,p}(\Omega)} \leq C \left( \|\D \phi\|_{L^p(\Omega)}+\left| \int_\omega \phi \,\dd x \right| \,\right),\]
and we estimate the second term in the right hand side by
\begin{equation*}
\begin{split}
\left| \int_\omega \phi \,\dd x \right| &\leq \int_\omega | \phi | \,\dd x \leq \int_\omega | \dist_2 \circ \phi | \,\dd x + |\omega|\sup_{x \in \M_2}|x|  \\
&\leq \int_\omega | \dist_2 \circ \phi - \dist_1 | \,\dd x + \int_\omega |\dist_1| \,\dd x + |\omega|\sup_{x \in \M_2}|x|  \\
&\leq \eta_\sigma\left(\frac{\sigma}{2}\right)^{-1} |\omega|^{-\frac{1}{2}} \left( \nu E_{\text{match}}[\phi]\right)^{\frac{1}{2}} + \int_\omega |\dist_1| \,\dd x + |\omega|\sup_{x \in \M_2}|x|,
\end{split}
\end{equation*}
where H\"{o}lder's inequality has been used to compare $L^1$ and $L^2$ norms. Therefore, \eqref{eq:modifiedEstimate} follows.

The proof of the estimate for $\| \dist_2 \circ \phi \|_{L^\infty(\{|\dist_1| \leq \sigma\})}$ (to ensure that deformations stay away from the singularities of $\dist_2$) is still valid with minor modifications, since $\nu$ appears in \eqref{eq:modifiedEstimate} multiplicatively.
\end{proof}

We conclude this section with the following proposition, which explores the penalization limit in which the parameter $\nu$ tends to zero.

\begin{proposition}\label{prop:penalization-limit}
Let $\{\nu_k\}_{k \in \mathbb{N}}$, be a sequence of penalty matching parameters such that $\nu_k \to 0$  as $k \to \infty$, and $\phi^k$ be solutions of the Dirichlet minimization problem for $E_{\nu_k}$. Then, up to a choice of subsequence, the $\phi^k$ converge strongly 
in $W^{1,p}$ to a minimizer of
\[E_{\text{mem}}+E_{\text{bend}}+E_{\text{vol}}\] 
in $W_0^{1,p}(\Omega;\R^n)+\Id$ under the constraint $\phi(\M_1^{c}) = \M_2^{c}$ for all $c \in (-\sigma, \sigma)$.
\end{proposition}
\begin{proof}
First, notice that the energy $E$ may be written as
\begin{eqnarray}\label{eq:ERewritten}
E_\nu[\phi]=&&\frac{1}{\nu}\int_\Omega \eta_\sigma \circ \dist_1 |\dist_2 \circ \phi - \dist_1|^2 + \alpha_p |\D\phi|^p\\
&&+H\Big(\det \D\phi, \Cof \D\phi, \Di{tg}\phi, \det ( \Di{tg}\phi + \n_2 \circ \phi  \otimes \n_1) , \nonumber\\
&&\Lambda(\mathcal{C}(\mathcal{S}^{ext}_1),\mathcal{C}(\mathcal{S}^{ext}_2 \circ \phi)) ,\D \phi),\det\big( \Lambda(\mathcal{C}(\mathcal{S}^{ext}_1),\mathcal{C}(\mathcal{S}^{ext}_2 \circ \phi)), \D \phi \big) \Big)\,\dd x,\nonumber
\end{eqnarray}
where $H:\R^+ \times \R^{n \times n} \times \R^{n \times n} \times \R \times \R^{n \times n} \times \R \to \R^+$ is smooth and convex.

Denote by $\hat{\varphi}$ the extension of a diffeomorphism between $\M_1$ and $\M_2$ used in the proof of Theorem \ref{thm:existence}. Since $E_{\text{match}}[\hat{\varphi}]=0$, we have that $E_{\nu_k}[\phi^k] \leq E_1[\hat{\varphi}]$. By the coercivity estimate \eqref{eq:coercivity} the $\phi^k$ are then bounded in $W^{1,p}$ and we may extract a (not relabelled) subsequence converging uniformly and weakly in $W^{1,p}$ to some limit $\phi$. Since $\{E_{\nu_k}[\phi^k]\}$ is bounded and $\nu_k \to 0$, the uniform convergence of $\phi^k$ implies that 
\begin{equation}\label{eq:conv-ematch}\int_\Omega \eta_\sigma(\dist_1)|\dist_2 \circ \phi^k - \dist_1|^2 \,\dd x \xrightarrow[k \to \infty]{} \int_\Omega \eta_\sigma(\dist_1)|\dist_2 \circ \phi - \dist_1|^2 \,\dd x = 0.\end{equation}
In consequence, $\phi(\M_1^c)\subseteq \M_2^c$. Since $\restr{\phi}{\M_1^c}$ is the uniform limit of the maps $\restr{\phi^k}{\M_1^c}$ which are surjective onto $\M_2^c$ and $\M_1^c$ is compact, we conclude that $\phi(\M_1^c)=\M_2^c$ for all $c \in (-\sigma, \sigma)$. Therefore, $\phi$ is admissible for all $\nu_k$ and $E_{\nu_k}[\phi^k] \leq E_{1}[\phi]$. Combined with lower semicontinuity and \eqref{eq:ERewritten}, the above implies
\begin{equation}\label{eq:conv-derpart}
\int_\Omega \alpha_p |\D\phi^k|^p + H(\det(\D\phi^k), \ldots ) \,\dd x \xrightarrow[k \to \infty]{} 
\int_\Omega \alpha_p |\D\phi|^p + H(\det(\D\phi), \ldots ) \,\dd x.
\end{equation}
From this identity, the fact that $H$ is convex and differentiable, and $\D\phi^k \rightharpoonup \D \phi$ in $L^p$ it follows that 
\begin{equation*}
 \begin{aligned}
  0&= \limsup_{k \to \infty} \left( \int_\Omega \alpha_p \left( |\D\phi^k|^p - |\D\phi|^p \right) + 
                     H(\det(\D\phi^k), \ldots ) - H(\det(\D\phi), \ldots ) \,\dd x \right) \\
   &\geq 
      \limsup_{k \to \infty} \left( \int_\Omega \alpha_p \left( |\D\phi^k|^p - |\D\phi|^p \right) + 
                     \D H(\det(\D\phi), \ldots ) \cdot (\det(\D\phi^k) - \det(\D\phi), \ldots ) \,\dd x\right)\\
   &= \limsup_{k \to \infty} \int_\Omega \alpha_p |\D\phi^k|^p \dd x - \int_\Omega \alpha_p |\D\phi|^p \,\dd x\;.
 \end{aligned}
\end{equation*}
Together with the weak lower semicontinuity of the $L^p$-norm, the above shows that 
\begin{equation*}
   \int_\Omega \alpha_p |\D\phi|^p \,\dd x = \lim_{k \to \infty} \int_\Omega \alpha_p |\D\phi^k|^p \dd x\;.
\end{equation*}
Because $L^p(\Omega)$ has the Radon-Riesz property (\cite[2.5.26]{Meg98}), weak convergence and convergence of the norm guarantee strong convergence.
Since $\phi_k$ was assumed to converge uniformly, we have also $\phi_k \to \phi$ in $L^p$, and this shows that 
$\phi_k \to \phi$ in $W^{1,p}(\Omega; \R^{n})$. 

That $\phi$ is a minimizer of the constrained problem follows directly (\cite{Bra02}, Theorem 1.21) from the fact that the $E_{\nu_k}$ are an equicoercive family of functionals, $\Gamma$-converging in the weak topology of $W^{1,p}$. Indeed, equicoercivity follows easily from the above, while $\Gamma$-convergence is implied by the fact that $E_{\nu_k}$ is an increasing sequence (\cite{Bra02}, Remark 1.40), because $\nu_k \to 0$ appears as a denominator in $E_{\text{match}}$. 
\end{proof}

\begin{remark}By the coercivity estimate \eqref{eq:modifiedEstimate} of Corollary \ref{cor:neumann}, an entirely analogous result holds for minimizers with Neumann boundary conditions.
\end{remark}

\begin{remark}
Contrary to what might be expected, the limit problem we have obtained is not a surface problem, since all the level sets are still coupled through the volume energy $E_{\text{vol}}$. The line of reasoning above depends heavily on the fact that the coefficients of the volume term are held fixed, since the equicoercivity and uniform strict quasiconvexity (in the language of \cite{EvaGar87}) both require the presence of $\|\D\phi\|^p_{L^p(\Omega)}$ in the functional.
\end{remark}
\subsection{Oscillations and lack of rank-one convexity for the naive approach}
To model the tangential distortion energy we have considered a frame indifferent energy density with the argument $\Di{tg}\phi +(\n_2 \circ  \phi ) \otimes \n_1$.
Let us now consider \added{the case $n=2$ and} a simpler version of the membrane energy \eqref{def:Ematch}, where we use as an argument of the energy density directly the tangential Cauchy-Green strain tensor (cf \eqref{eq:Ctt})
$(\DiOs \phi(x))^T (\DiOs \phi(x)) + \n_1(x) \otimes \n_1(x)$,
and define the membrane energy
\begin{equation}\label{eq:badMem}
\tilde E_\text{mem}[\phi]:=\int_{\Omega} \eta_\sigma (\dist_1(x))  W\!\! \left( \big(\DiOs \phi(x)\big)^T \DiOs \phi(x) + \n_1(x) \otimes \n_1(x) \right) \dd x,
\end{equation}
with $\DiOs\phi := \D\phi \P_1$ defined as the tangential part of the derivative along $T_x \M_1^{\dist_1(x)}$, 
and $W:\R^{2 \times 2}\to\R$ a frame indifferent energy density that has a strict minimum at $\SO(2)$. 
In fact, this energy is no longer lower semicontinuous and we will present counterexamples.
\begin{example}[Oscillation patterns]
We construct an explicit sequence for which lower semicontinuity of the membrane energy $\tilde E_\text{mem}$ fails. Fix $0<R<1$ and $\M_1 = \S^1$ with the parametrization $\xi \to e^{i \xi}$. Consider a sequence of deformations $\varphi_k: \S^1 \to \R^2$ defined in polar coordinates of $(r, \theta)$ by the condition
\begin{equation}\label{eq:varphik}
\partial_\xi \varphi_k(\xi) = \left(R \sin k\xi \right) e_r\big( r(\varphi_k( \xi ) ), \theta(\varphi_k( \xi ) ) \big) + \left(1-R^2 \sin^2 k \xi \right)^{\frac{1}{2}} e_\theta \big(r(\varphi_k( \xi )), \theta(\varphi_k( \xi ) ) \big),
\end{equation}
 where $e_r=(\cos \theta, \sin \theta)^T, e_\theta=(-\sin \theta, \cos \theta)^T$ for given $\phi_k(0)$. 
 Note that for any $k$ and $\theta$ that $|\partial_\theta \varphi_k(\theta)|=1$, so that the transformations are tangentially isometric. 
We define $\varphi_k(0)$ via two integration constants $r_0$ and $\theta_0$ for the initialization of $r$ and $\theta$ at $\xi=0$. 
We set $\theta_0=0$ and choose $r_0$  such that the curve $\varphi_k$ is closed and simple, which imposes $r_0=r(\varphi_k(0))=r(\varphi_k(2\pi))$ 
since the first term in \eqref{eq:varphik} has zero average. From the second term, taking into account that $e_\theta(r, \theta)$ is independent of $r$, we get the condition
\[2 \pi r_0 = \int_0^{2\pi}\left(1-R^2 \sin^2 k \xi \right)^{\frac{1}{2}} \dd \xi=\frac{1}{k} \int_0^{2\pi k}\left(1-R^2 \sin^2 \zeta \right)^{\frac{1}{2}} \dd \zeta,\]
where we have applied the change of variables $\zeta = k \xi$. By periodicity the right hand side (an incomplete elliptic integral of the second kind with modulus $R$) is independent of $k$ and thus determines $r_0$. The resulting $\varphi_k$ for several values of $k$ are depicted in Figure \ref{fig:oscillation}. 

We observe that
$\partial_\theta \varphi_k(\theta) \rightharpoonup r_0 e_\theta \text{ in }L^p,$
for any $1 \leq p < \infty$ (and also weak-* in $L^\infty$). Therefore, the weak $W^{1,p}$-limit $\varphi$ of the $\varphi_k$ is the function defined by $\varphi(\theta)=r_0 e_r$ and obviously not an isometry. 
Assuming $0 < \sigma < 1$ and extending $\varphi_k, \varphi$ along the radial direction $e_r$ to the annulus $\{1-\sigma \leq r \leq 1+\sigma\}$, we obtain corresponding deformations given by 
\[\phi^k(r,\theta)=\varphi_k(\theta)+(r-1)Q_{\frac\pi2}\partial_\theta \varphi_k(\theta) \text{, and }\phi(r,\theta)=\varphi(\theta)+(r-1)r_0 e_r=r \, r_0 e_r,\]
where $Q_{\frac\pi2}$ stands for clockwise rotation by $\pi/2$, so that $Q_{\frac\pi2}\partial_\theta \varphi_k(\theta)$ is the unit outward normal to $\varphi_k(\S^1)$. Clearly also $\phi^k \rightharpoonup \phi$ in $W^{1,p}$ on the annulus. We observe that $\tilde E_{\text{mem}}[\phi^k]=0$, but $\tilde E_{\text{mem}}[\phi]>0$. Hence,  $\tilde E_\text{mem}$ is not weakly lower semicontinuous.
\begin{figure}[ht]
    \begin{center}
        \includegraphics[width=0.85\textwidth]{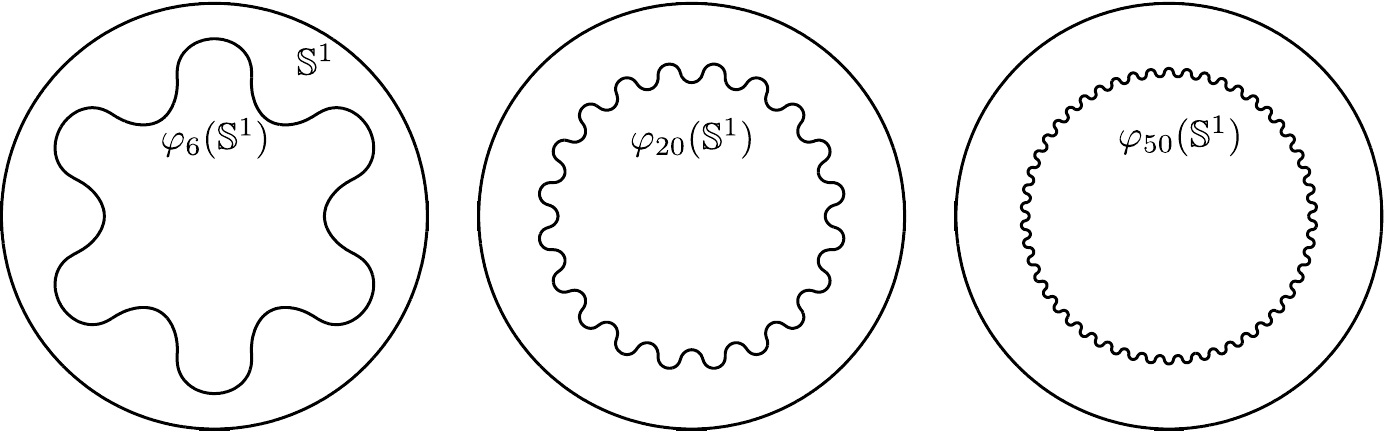}
    \caption{Explicit oscillations for a simplified model. $\varphi_k$ for $R=0.95$, $k=6,20,50$}
    \label{fig:oscillation}
    \end{center}
\end{figure}
\end{example}
The celebrated Nash-Kuiper theorem \cite{Nas54, Kui55} states that it is possible to uniformly approximate any short $C^\infty$ immersion by $C^1$ isometric ones. Our explicit oscillations around $r_0 \S^1$ is just one example of this phenomenon.  Notice that a bending term of the type $E_{\text{bend}}$ introduced in our model only compares the curvatures of $\M_1^{\dist_1(x)}$ and $\M_2^{\dist_2(\phi(x))}$. It therefore does not penalize oscillations, since it does not detect the curvature of $\phi(\M_1)$ at all.

\begin{example}\label{ex:norank1conv}[Lack of rank-one convexity]
We present an additional example of a configuration for which the integrand of an energy of the type $\tilde E_{\text{mem}}$ is not rank-one convex. 
Rank-one convexity of the complete energy density, \ie, convexity in $t \in \R$ when composed with the function $A+tB$ for any matrix $A$ and any rank one matrix $B$, is known to be a necessary condition for quasiconvexity 
(\cite{Dac08}, Theorem 5.3). Quasiconvexity, in turn, is necessary for weak lower semicontinuity of integral functionals in Sobolev spaces (\cite{Dac08}, Theorem 8.1 and Remark 8.2).

Let $\Omega=\added{(-2,2)^2}$, and \added{$\M_1$ be a closed $C^2$ curve such that $\M_1 \cap (-1,1)\times(0,2)=(-1,1)\times\left\{\added{1}\right\}$. At any point $x_0 \in (-1,1)\times\left\{1\right\}$}, the tangential derivatives are just partial derivatives along the first coordinate, yielding
\begin{gather*}\DiOs\phi\added{(x_0)} = \D\phi\added{(x_0)} \P(e_2)=\left(\begin{array}{cc} \partial_1 \phi_1\added{(x_0)} & 0 \\ \partial_1 \phi_2\added{(x_0)} & 0 \end{array}\right), \textrm{ and }\\
\DiOs\phi\added{(x_0)}^T \DiOs\phi\added{(x_0)} = \left(\begin{array}{cc} \left(\partial_1 \phi_1\added{(x_0)}\right)^2 + \left(\partial_1 \phi_2\added{(x_0)}\right)^2 & 0 \\ 0 & 0 \end{array}\right).\end{gather*}
Hence the tangential area distortion measure reduces to
\begin{equation}\begin{aligned}\label{eq:ad1d}\tr(\DiOs\phi\added{(x_0)}^T \DiOs\phi\added{(x_0)})&=\det(\DiOs\phi\added{(x_0)}^T \DiOs\phi\added{(x_0)}+e_2 \otimes e_2)\\
&=\left(\partial_1 \phi_1\added{(x_0)}\right)^2 + \left(\partial_1 \phi_2\added{(x_0)}\right)^2,\end{aligned}\end{equation}
where $e_2=(0,1)^T$. Defining now the convex function 
\[F(a,d)=\frac{1}{2} a + \frac{1}{2} d + d^{-1}-2,\]
which has a unique minimum with value $0$ for $a=d=1$, we have that the energy density
\[W_F(B)=F\left( \tr(B^TB), \det(B^TB+e_2 \otimes e_2) \right)\]
has a pointwise minimum, with value zero, whenever $\D\phi$ is such that $\left(\partial_1 \phi_1\right)^2 + \left(\partial_1 \phi_2\right)^2=1$. 

Consider now, for $0\leq \lambda \leq 1$, the family of matrices
\begin{equation}\label{eq:rank1conn}B(\lambda)=\left(\begin{array}{cc} \lambda & 0 \\ (1-\lambda) & 0 \end{array}\right)=\lambda \left(\begin{array}{cc} 1 & 0 \\ 0 & 0 \end{array}\right)+(1-\lambda) \left(\begin{array}{cc} 0 & 0 \\ 1 & 0 \end{array}\right).\end{equation}
Clearly $B(\lambda)$ is rank one. But we have
$W_F(B(\lambda))=\lambda^2+(1-\lambda)^2+\frac{1}{\lambda^2+(1-\lambda)^2}-2$ and therefore
\[W_F(B(0))=F(B(1))=0, \textrm{ but }W_F( B(1/2))=\frac{1}{2},\]
which demonstrates that $W_F$ is not rank-one convex. 
%
\end{example}
 

\section{Finite element discretization based on adaptive octrees}\label{sec:discr}
We adopt a `discretize, then optimize' approach and consider a finite element approximation and optimize for the coefficients of the solution. Since the energy $E_\nu$ is highly nonlinear and nonconvex, we use a cascadic multilevel minimization scheme in which the solution for one grid level is used as the initial data for the minimization on the next finer grid. We use adaptive refinement of the underlying meshes around the surfaces $\M_1, \M_2 \added{\subset (0,1)^n}$ \added{for $n=2,3$} (Algorithm \ref{alg:scheme}).

One of the main characteristics of our functional is the pervasive presence of coefficients depending on the deformed position $\phi(x)$. Indeed, this is how the functional takes into account the geometry of target surface, through the projection $\P_2$ and shape operator $\mathcal{S}_2$. From an implementation perspective, however, this means that frequently discrete functions have to evaluated at deformed positions. Therefore, the ability to efficiently search the index of an element containing a given position is of paramount importance, so a hierarchical data structure that allows for efficient searching is needed. The model only contains first derivatives of the unknown deformation. Hence, multilinear finite elements already allow a conforming discretization. 
For these reasons we use multilinear FEM on octree grids. The grids used are such that all of the elements are either squares or cubes of side length $h=2^{-\ell}$, for an integer $\ell$ to which we refer as grid level of the element. In what follows let us detail the different ingredients of the algorithm.

\begin{algorithm}[H]
\begin{minipage}{0.5\columnwidth}
  \begin{algorithmic}
		\State Starting grid: Uniform of level $\ell_{\text{min}}$, $h=2^{-\ell_{\text{min}}}$
		\State $\phi$ $\gets$ $\Id$
    \For{$\ell \gets \ell_{\min} \textrm{ to } \ell_{\max}$}
		\State Regenerate $\dist_1, \dist_2$ on grid by aFMM.
		\State Compute $\n_1$, $\n_2$, $\mathcal{S}_1$, $\mathcal{S}_2$ from $\dist_1, \dist_2$.
		\State $\phi \gets$ $L^2$-CG-descent ($\phi$)
			\State Mark all elements intersecting $\M_1$ or $\M_2$.
			\State Refine the grid ($h$ $\gets$ $2^{-\ell+1}$).
    \EndFor
    \State \Return{$\phi$}
  \end{algorithmic}
  \caption{\label{alg:scheme}Cascadic minimization scheme.}
  \end{minipage}
\end{algorithm}

\paragraph*{\textbf{Multilinear Finite Elements on Octrees.}}
We assume $n=3$ for the presentation here. 
Using an adaptive octree grid based on cubic cells leads to hanging nodes (see Figure \ref{fig:grids}), nodes which are on the facet of a cell without being one of its vertices. Enforcing continuity of the finite element functions leads to constraints for function values on hanging nodes and these hanging nodes are not degrees of freedom. Additionally, to minimize the complexity of the required interpolation rules, the subdivision is propagated in such a way that the grid level of neighboring elements sharing a cell facet differs at most by one.
\paragraph*{\textbf{Octrees and the access to degrees of freedom via hashtables.}}
Even though the tree structure gives a natural hierarchical structure to the elements of the mesh, maintaining consistent linear indices for degrees of freedom, hanging nodes, and elements can be delicate. Consistent rules could be devised to maintain consistency with the element octree for a given mesh, but these would not be easy to update when the grid is refined. 
In order to keep track of vertex indices in a simple manner without sacrificing efficiency, hash maps (\cite{ColeRiSt09}, Chapter 11) are maintained to keep track of the indices of degrees of freedom, hanging nodes, and cells. The keys used in the hashmap are a combination of a level value $\ell$ and point coordinates as integer multiples of $h=2^{-\ell}$. These keys uniquely identify nodes or elements, with the convention that an element is identified with its lower-left-back corner. Whenever a query for a node or cell is made, there are two possible outcomes. If it is already contained in the corresponding hash table, a linear index for it 
can be retrieved. Otherwise, a new entry of the hash table is created and the node or cell is given the next unused index. Since we do not require coarsening of the mesh, this scheme guarantees a consistent linear set of indices with a computational cost for insertions and queries that is, on average, independent of the mesh size.
\paragraph*{\textbf{Computing distance functions on octrees.}}
In our model we have assumed that the distance functions to our surfaces are given. In practice, especially when using adaptive grids, we need to compute signed distance functions on such grids. This has been accomplished by a straightforward adaptation of the Fast Marching Method on cartesian grids \cite{Set99} exploiting the fact that our grids still are subgrids of a regular cartesian grid. In the implemented variant hanging nodes are not taken into account for the propagation, their values being linearly interpolated to accommodate the constraints needed for conformality. 
The initialization for the distance computation has been performed starting from triangular meshes of the surfaces (for $n=3$; for $n=2$ two-bit segmentation of interior and exterior of the curves has been used). 
The signs of the distance functions have to be computed separately, by detecting which points of the grid are inside (resp. outside) the initial surface data. In our case, they have been computed with the provably correct algorithm given in \cite{BaAa05}.
\paragraph*{\textbf{Computation of the coefficients.}}
The discretization for the unknown deformation $\phi$, as already mentioned, is done by multilinear finite elements. However, the coefficients of our model include first and second derivatives of the signed distance functions $\dist_i$, for the normal vectors $\n_i$ and shape operators $\mathcal{S}_i$ ($i=1,2$), respectively. The approximations are required to be robust, since they appear in the highest order terms of the model. 
For the normal vectors $\n_i$, we compute the $L^2$ projection of the finite element derivative of $\dist_i$ to recover the nodal values of a piecewise multilinear function, followed by a orthogonal projection to the unit sphere to restore the constraint $| \n_i | = 1$.

In the case of the shape operators, our approach is to approximate the distance functions $\dist_i$ by a quadratic polynomial supported on a neighborhood of each point. Given a fixed integer neighborhood size $r$, for each non hanging node $x_k$ (\ie the neighborhood $B_r(x_k)$ contains the $r$ closest other degrees of freedom of the adaptive grid) the local quadratic polynomial  $p_k$ is defined as the one minimizing the least-squares error
\[\sum_{x_j \in B_r(x_k)} \left(p_k(x_j)-\dist_i(x_j) \right)^2\,.\]
which can be easily computed by inverting a small matrix. The Hessian of $\dist_i$ at the node $x_k$ is then approximated by the Hessian of $p_k$.

For the computation of matrix square roots and their inverses, we have used the method described in \cite{Fra89}, taking appropriate care to truncate almost-singular matrices, since the resulting square roots also appear inverted.

\begin{figure*}
  \centering
  \mbox{} \hfill
  \includegraphics[width=.40\linewidth]{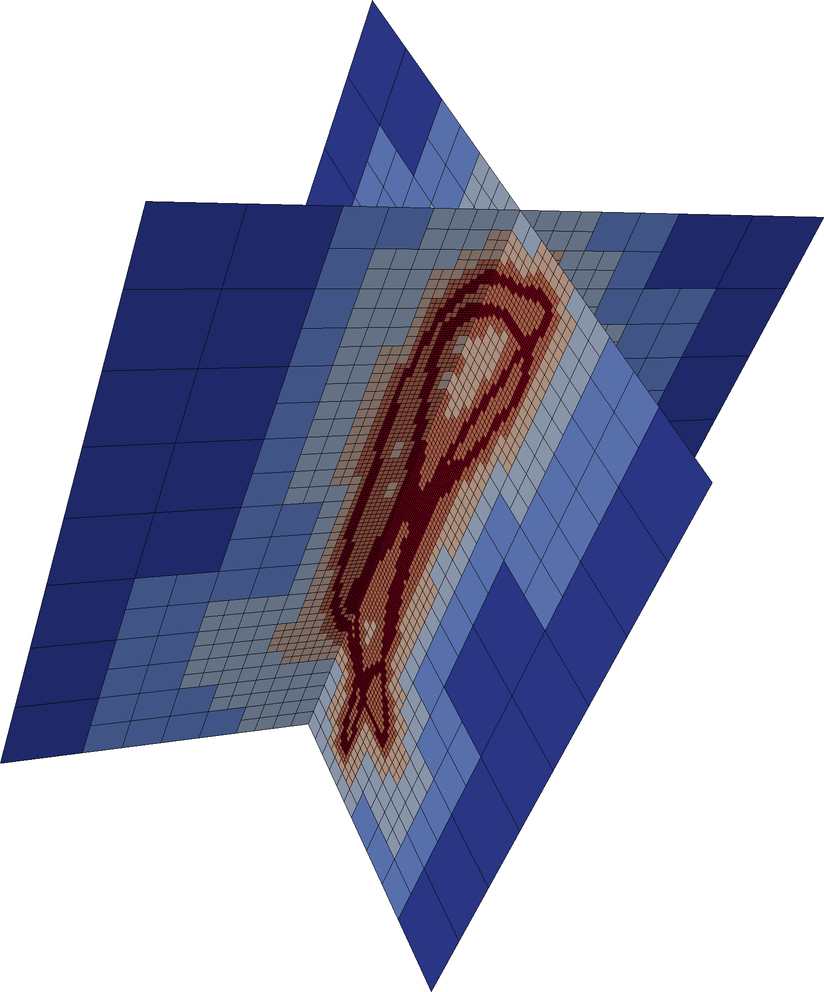}
  \hfill
  \includegraphics[width=.40\linewidth]{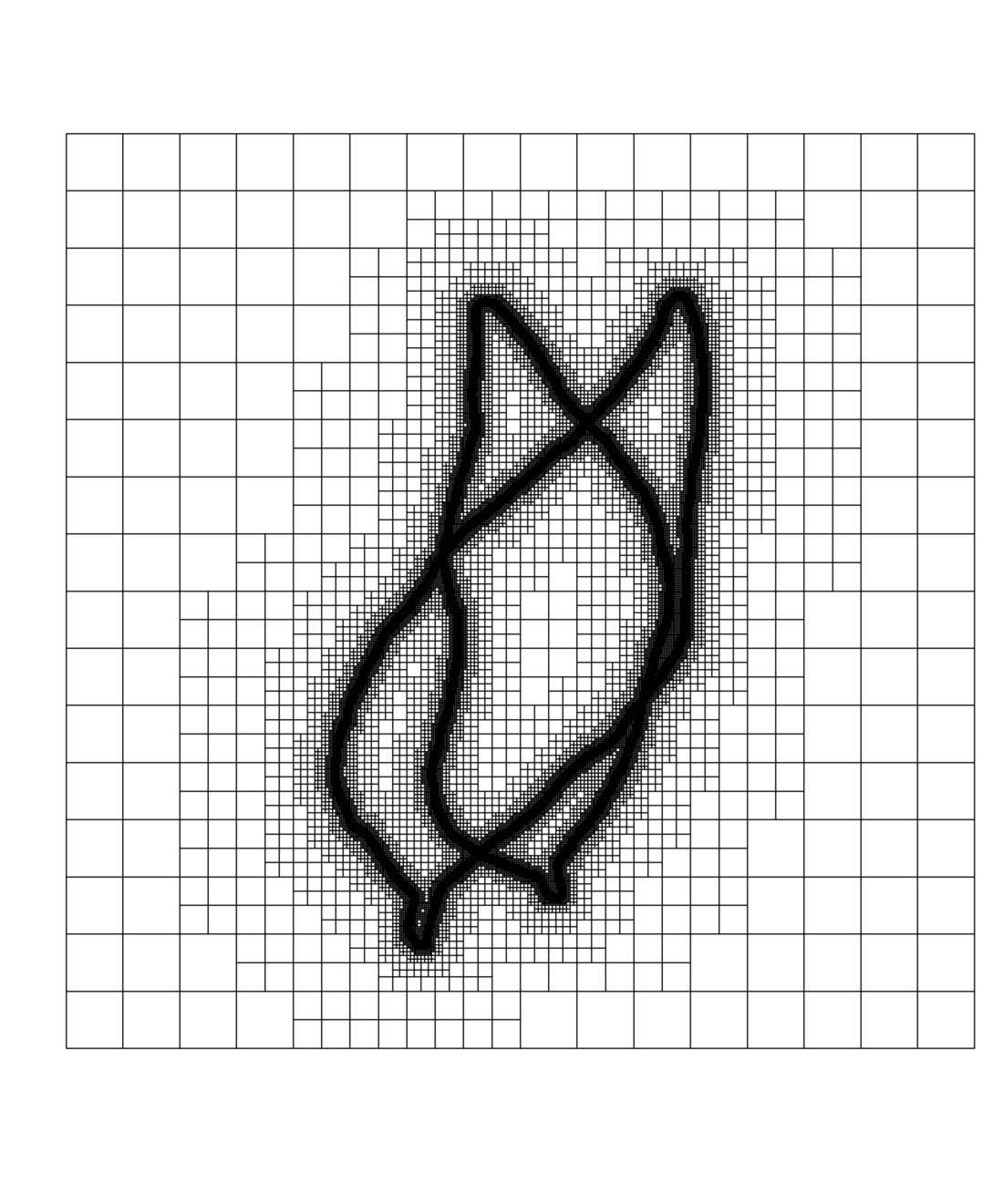}
  \hfill \mbox{}
  \caption{\label{fig:grids}Hierarchical grids corresponding to the dolphin surfaces (different 2D slices in 3D, grid level 8, 178584 DOFs, $1.1\%$ of the amount of DOFs in the full grid case) and leaf contours
  (2D, level 10).}
\end{figure*}

\paragraph*{\textbf{Minimization strategy.}}
For the minimization at each level, we have opted for a Fletcher-Reeves nonlinear conjugate gradient method (\cite{NocWri06}, Section 5.2).
The $L^2$ gradient of $E_\nu$, whose computation is involved but elementary, was implemented directly. The parameter $\alpha$ is progressively reduced when a further feasible descent step is not found, according to an Armijo line search (\cite{NocWri06}, Section 3.1).

\section{Numerical results}\label{sec:appl}
All of our results have been computed on the unit cube $\Omega=[0,1]^3$ for the matching of surfaces in 3D, and the unit square $[0,1]^2$ for the matching of contour curves in 2D. 
In practice, we have used homogeneous Neumann boundary conditions, since this allows to have relatively large shapes $\M_i$ in comparison to the size of the domain $\Omega$ without creating excessive volume energies near the boundary (for the justification we refer to Corollary \ref{cor:neumann}). However, if the boundary is not fixed, the deformed domain $\phi(\Omega)$ is not necessarily contained in $\Omega$, so evaluation of coefficients on deformed positions has to be appropriately handled numerically. We use a projection of outside position onto the boundary of $\Omega$ for sufficient large $\Dist(\M_2, \partial \Omega)$. 

For the membrane and the bending energy we use the material parameters $\lambda=\mu=1$, corresponding to the density
\[\W(A)=\frac{1}{2}| A |^2 + \frac{1}{4} (\det A)^2 + \frac{3}{2}e^{- ( \det A - 1 ) } - \frac{13}{4}.\]
In the bending term, the shape operators have been regularized through the truncated absolute value function with $\tau=1$. Since we work on the unit cube, this corresponds to a comparatively large curvature radius. For the volume term, given that enforcing orientation preservation in a finite element framework is a far from straightforward, it is advantageous to work with the simplified version
\[c_{\text{vol}}\int_\Omega W(\D \phi) \,\dd x.\]

We have run the minimization scheme of Algorithm \ref{alg:scheme} beginning from a uniform grid of level $\ell_{\text{min}}=2$ or $\ell_{\text{min}}=3$ with $9^3=729$ nodes, and refined up to $\ell_{\text{max}}=8$ for 3D examples. For 2D cases a reasonable range turned out to be $\ell_{\text{min}}=4, \ell_{\text{max}}=10$. The finest grids used for two of the examples below are depicted in Figure \ref{fig:grids}. The width  of the narrow band was chosen proportional the finest resolution of the mesh ($\sigma=2h$) since a small value of $\sigma$ clearly produces inaccurate results when $\eta_\sigma$ is evaluated on coarse grids. However, the constraint $\int_\Omega \eta_\sigma = 1$ ensures that the overall strength of the surface terms $E_{\text{match}}$, $E_{\text{mem}}$ and $E_{\text{bend}}$ is not affected. The value of the penalty constraint $\nu$ was divided by $10$ for each grid refinement, which is justified by Proposition \ref{prop:penalization-limit}. Furthermore, the 
volume weight $c_{\text{vol}}$ was also halved per level to allow for simultaneously higher initial regularization and close final matches. Note that this reduction is much slower than that of the matching parameter.

In all examples, we have used the identity as the initial deformation. It should be noted that although the energy is geometric by design, we are using a first-order descent method for its minimization. In consequence, an adequate rigid pre-alignment can be beneficial for intricate shapes. Figure \ref{fig:dolphin} shows results for the matching of two different dolphin shapes. Our variational approach is highly nonlinear and non-convex. Thus, the numerical approximation of the globally optimal deformation depends on the initialization of the deformation. Figure  \ref{fig:dolphinRot} shows that
the identity deformation as the initial deformation is advisable only if the expected optimal deformation is not too large. This is demonstrated by applying different rigid body motions to $\M_1$.

All figures have been produced by deforming the input data (polygonal curve or triangulated surface) via the resulting deformation $\phi$. This is in contrast to deforming the grid and plotting the resulting extracted level sets (which effectively visualizes the {\it inverse} deformation), as commonly done in the registration literature, and also in \cite{IgBeRuSc13}.

\paragraph*{\textbf{Test case.}}
First we present a simple test case to underline the qualitative properties of our model. Figure \ref{fig:cubes} shows a configuration in which a high amount of compression, combined with rotation, is required. Our model finds the intuitively correct deformation, but oscillations typical for the lack of lower semicontinuity of the underlying energy are induced when $\P_2$ is not used in the membrane and bending terms. The bending term assists in matching the curvatures even if the deformation is not rigid. Note, however, that for the optimal match the curvature energy $E_{\text{bend}}$ is not expected to vanish, as can easily be seen from \eqref{eq:bendingExplanation}, \eqref{eq:matchingObjective} and the related discussion in Section \ref{sec:levelsets}.
\begin{figure*}
  \centering
  \mbox{} \hfill
  \includegraphics[width=.243\linewidth]{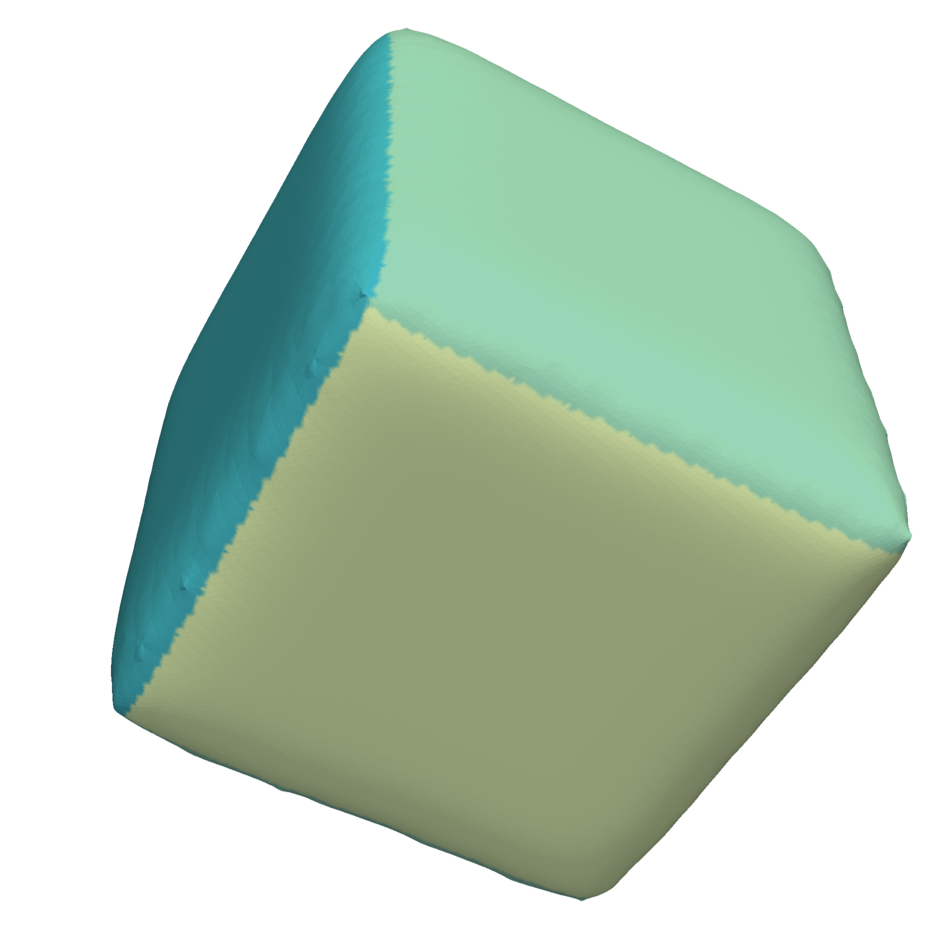}
  \hfill
  \includegraphics[width=.243\linewidth]{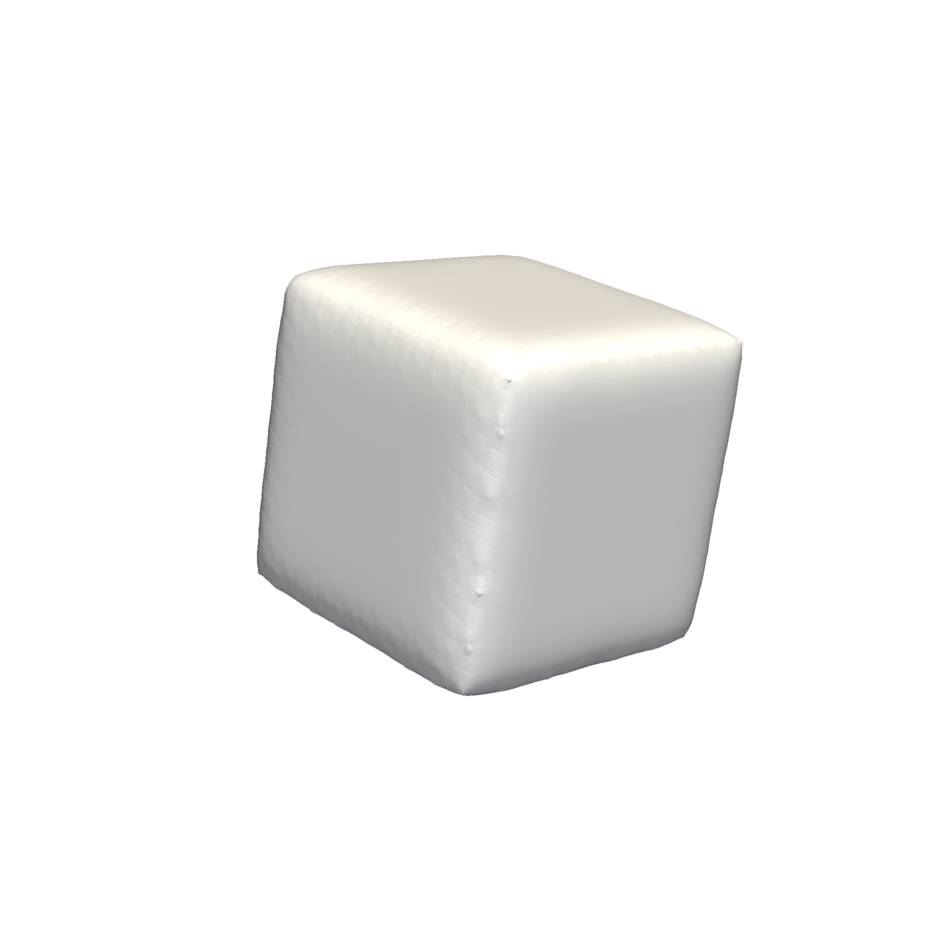}
  \hfill
  \includegraphics[width=.243\linewidth]{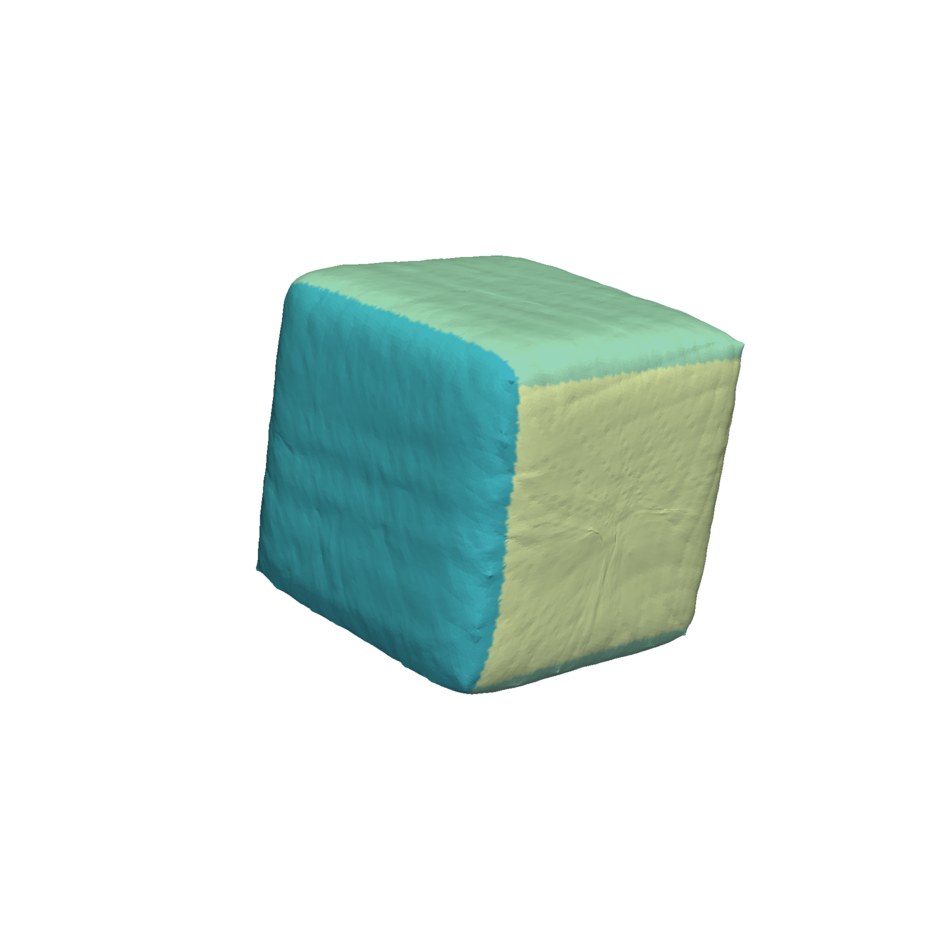}
  \hfill
  \includegraphics[width=.243\linewidth]{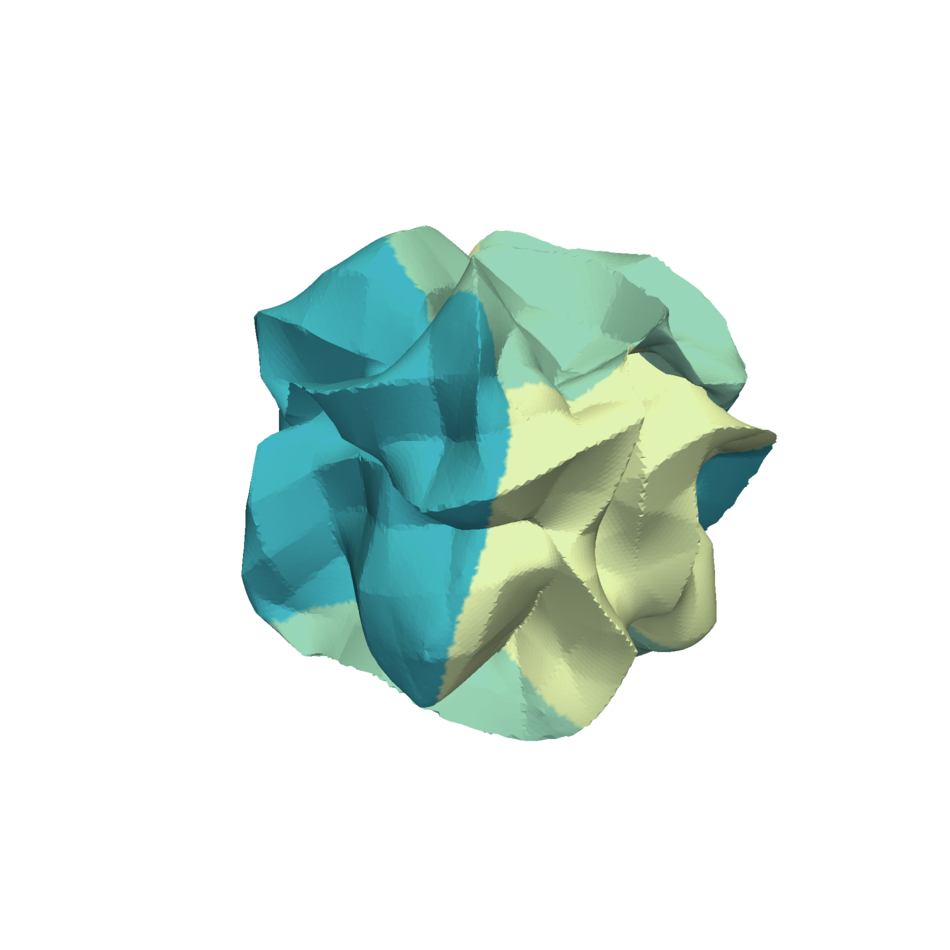}
  \hfill \mbox{}
  \caption{\label{fig:cubes}Behaviour of the optimal (numerical) deformation in the presence of strong compression. From left to right: Textured $\M_1$, $\M_2$, resulting deformed shape $\phi(\M_1)$ after grid level $7$ with our model, and corresponding result after grid level $4$ when $\P_2$ is not present in $E_{\text{mem}}$.}
\end{figure*}
\paragraph*{\textbf{Shape matching applications.}}
We now turn our attention to high resolution examples with real data. Figure \ref{fig:faces} demonstrates the effect of the multilevel descent scheme, in which details are added progressively to avoid spurious local minima. In Figure \ref{fig:leaves} a high-resolution 2D example is presented. Figures \ref{fig:hand}, \ref{fig:dolphin} and \ref{fig:beets} show 3D examples in which the influence of the curvature matching is indispensable to obtain shape sensitive matching deformations. For these examples, the shell parameter $\delta$ was chosen 
quite high, since the curvature matching term $E_{\text{bend}}$ is a major driving force to obtain correct matching of geometric features. Table \ref{tab:parstimes} lists the parameter values used, and run times for our implementation. We have split the timings between the highest-resolution level and the combined previous ones, since in many applications a very high level of detail might not be necessary, thereby significantly reducing the required computational effort.

\begin{table}[h!]
  \centering  
  \begin{tabular}{c|c|c|c|c|c|c}
    \toprule
    \multicolumn{1}{c}{Fig.} & 
    \multicolumn{1}{c}{$\ell_{\text{min}}, \ell_{\text{max}}$} & 
    \multicolumn{1}{c}{$\delta$} & 
    \multicolumn{1}{c}{$c_{\text{vol}}, \nu$ at $\ell_{\text{min}}$} &
    \multicolumn{1}{c}{Time, $\ell \leq (\ell_{\text{max}}-1)$} & 
    \multicolumn{1}{c}{Time, $\ell = \ell_{\text{max}}$} & 
    \multicolumn{1}{c}{DOFs at $\ell_{\text{max}}$}\\
    \midrule
    \ref{fig:faces}   & $3,8$ & 0.5  & $0.025, 0.002$ & 1h 04m  & 4h 34m  & 695K\\
    \ref{fig:hand}    & $2,8$ & 0.71 & $0.05, 0.1$    & 30m 10s & 1h 27m  & 313K\\
    \ref{fig:dolphin} & $3,8$ & 1    & $0.025, 0.002$ & 20m 04s & 50m 50s & 179K\\
    \ref{fig:beets}   & $3,8$ & 0.5  & $0.025, 0.002$ & 28m 56s & 1h 25m  & 408K
  \end{tabular}\vspace{0.3cm}
    \caption{Parameters and running times on a workstation with a single Intel Xeon E5-1650 CPU (6 cores, 3.2Ghz). Our implementation splits the computation of the different terms of the energy and the corresponding derivatives in different threads (obtaining a speedup factor $\approx 2$), but no further parallelization is used.}\label{tab:parstimes}
\end{table}

\begin{figure*}
  \centering
  \mbox{} \hfill
  \includegraphics[width=.190\linewidth]{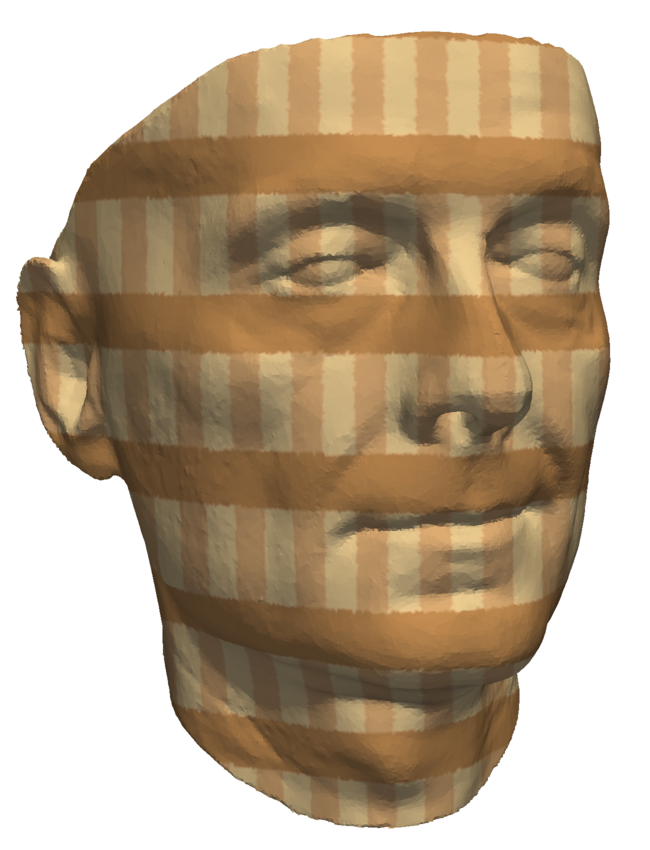}
  \hfill
  \includegraphics[width=.190\linewidth]{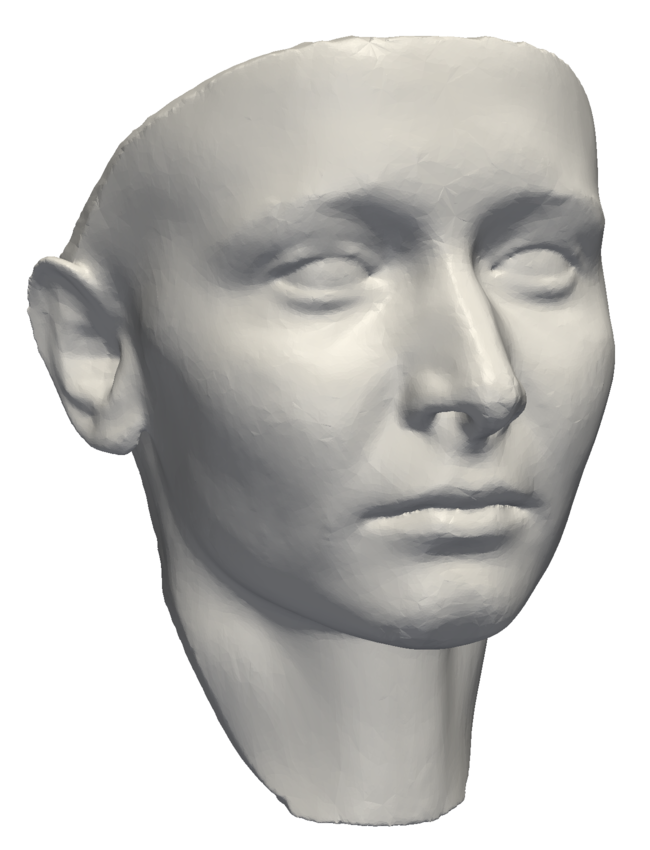}
  \hfill
  \includegraphics[width=.190\linewidth]{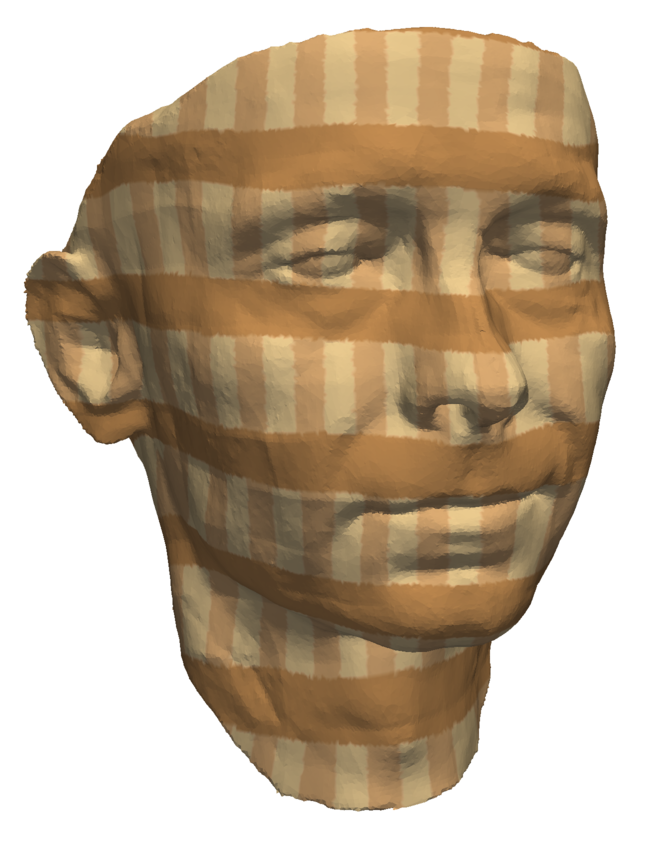}
  \hfill
  \includegraphics[width=.190\linewidth]{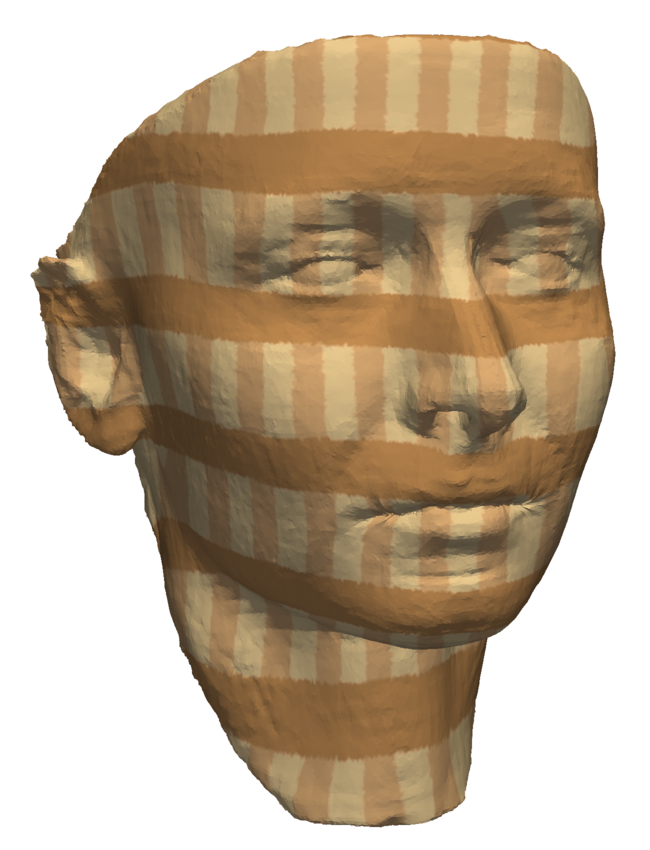}
  \hfill
  \includegraphics[width=.190\linewidth]{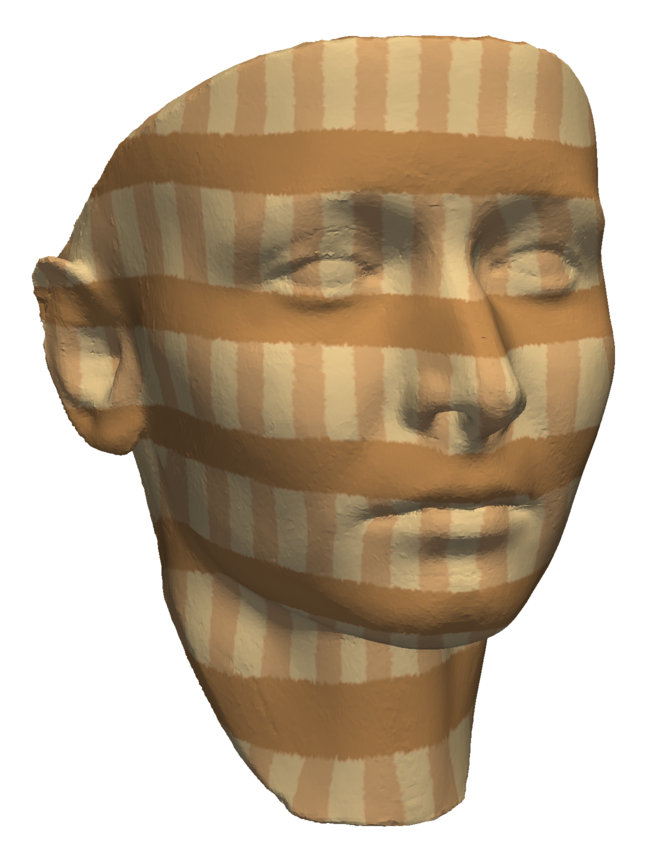}
  \hfill \mbox{}
  \caption{\label{fig:faces}Detail is added progressively in the cascadic coarse-to-fine scheme. From left to right: Textured $\M_1$, $\M_2$, resulting deformed shape $\phi(\M_1)$ after the computation on grid level $4,6$ and $8$, respectively.}
\end{figure*}
\begin{figure*}
  \centering
  \mbox{} \hfill
  \includegraphics[width=.22\linewidth]{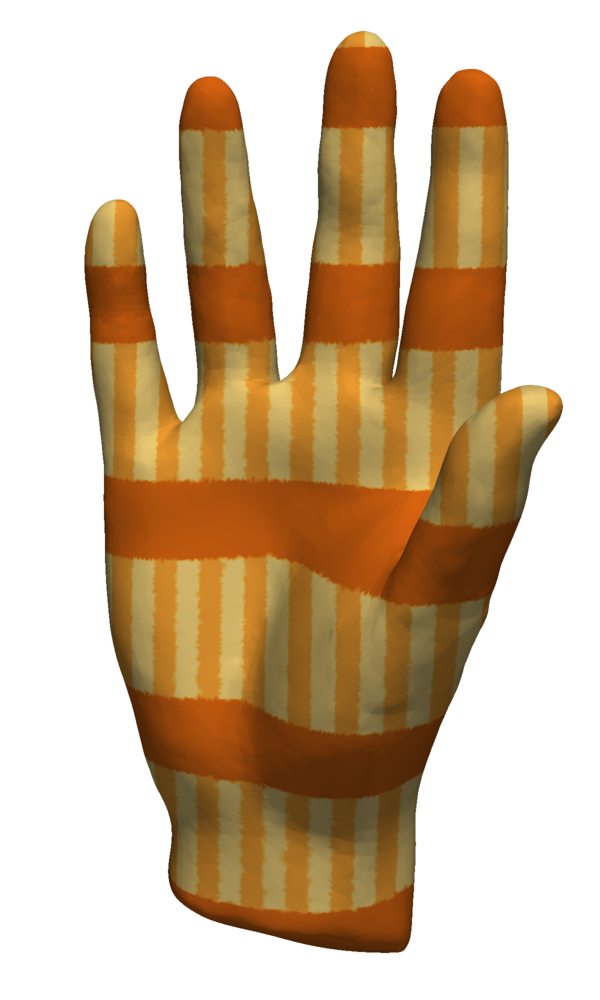}
  \hfill
  \includegraphics[width=.22\linewidth]{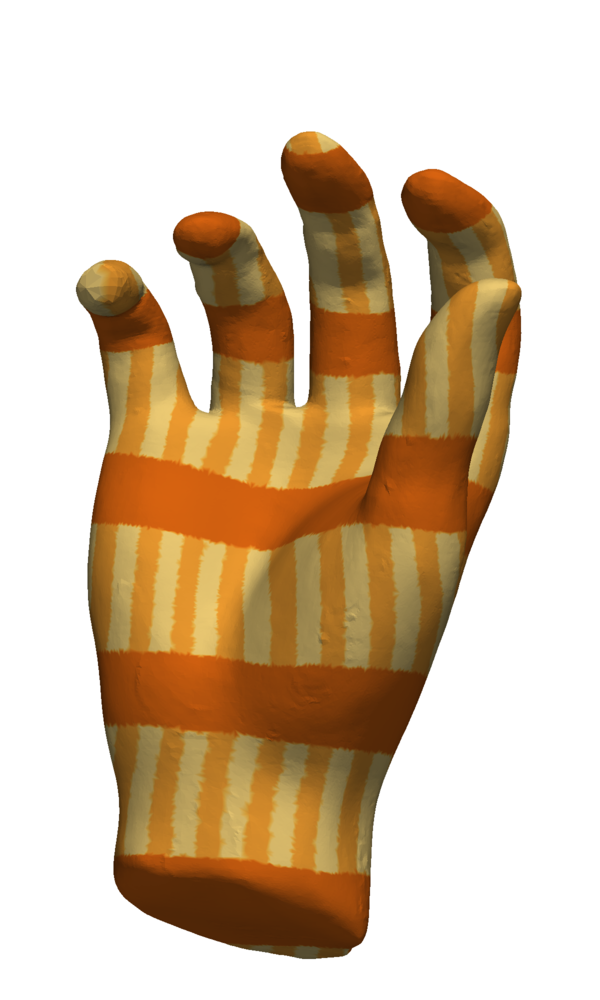}
  \hfill
  \includegraphics[width=.22\linewidth]{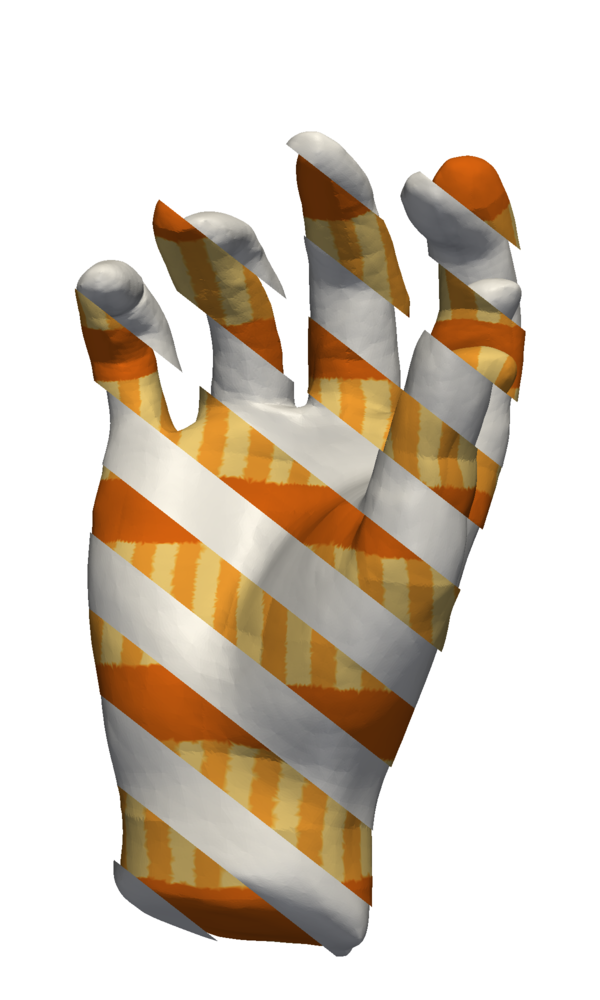}
  \hfill
  \includegraphics[width=.22\linewidth]{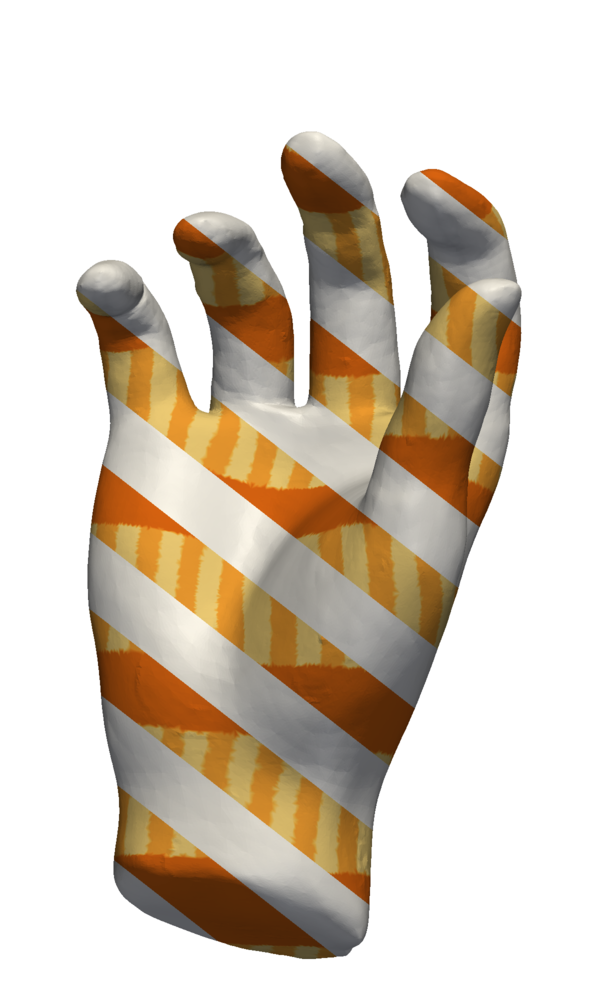}
  \hfill \mbox{}
  \caption{\label{fig:hand}From left to right: Textured hand shape $\M_1$, resulting deformed shape $\phi(\M_1)$ after level $8$ in the minimization scheme, comparison of target and obtained shapes after the computation on grid level $4$ and $8$, respectively.}
\end{figure*}
\begin{figure*}
  \centering
  \mbox{} \hfill
  \includegraphics[width=.192\linewidth]{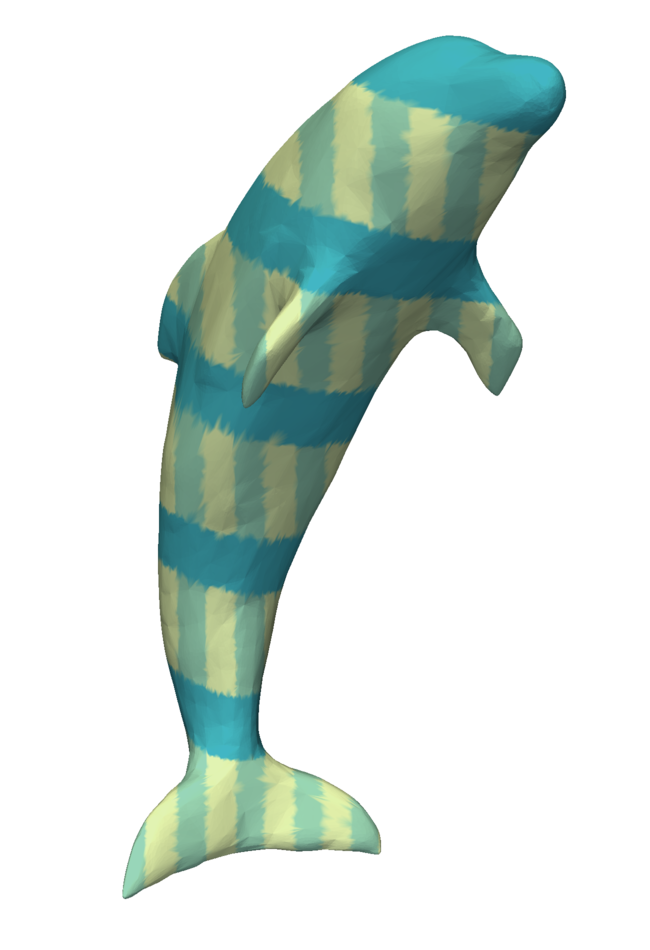}
  \hfill
  \includegraphics[width=.192\linewidth]{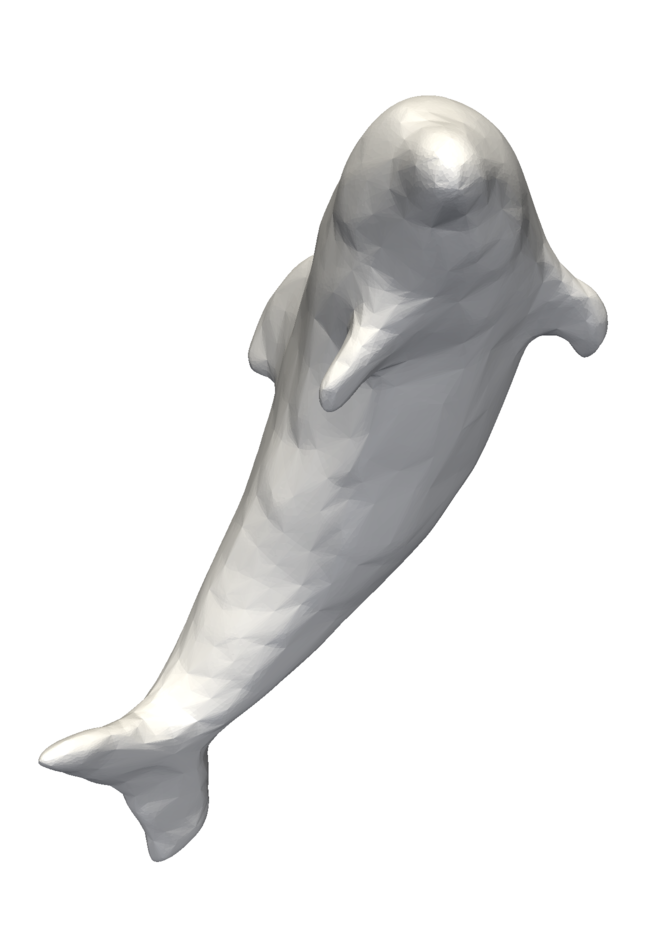}
  \hfill
  \includegraphics[width=.192\linewidth]{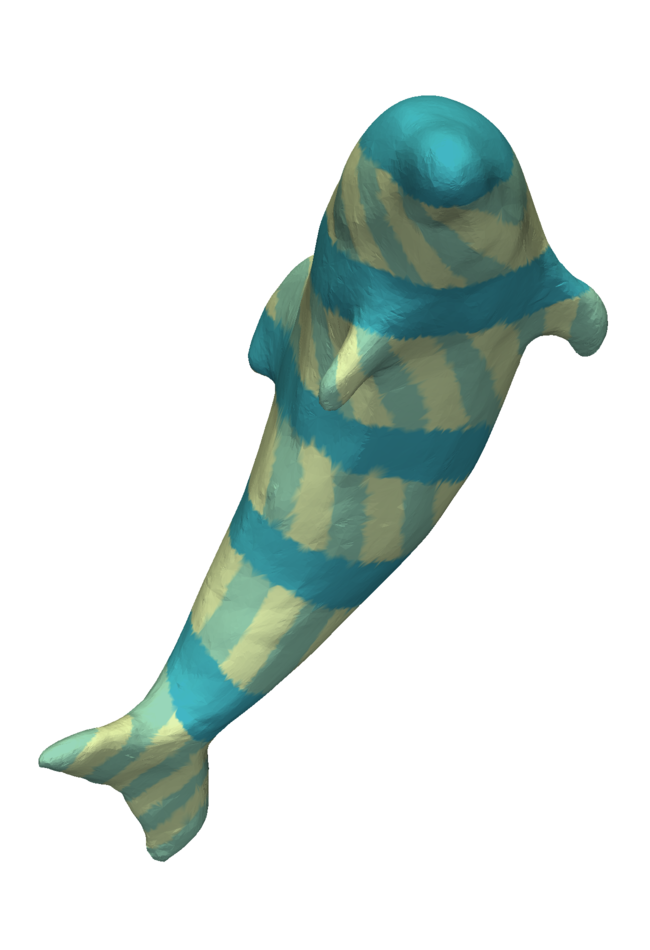}
  \hfill
  \includegraphics[width=.192\linewidth]{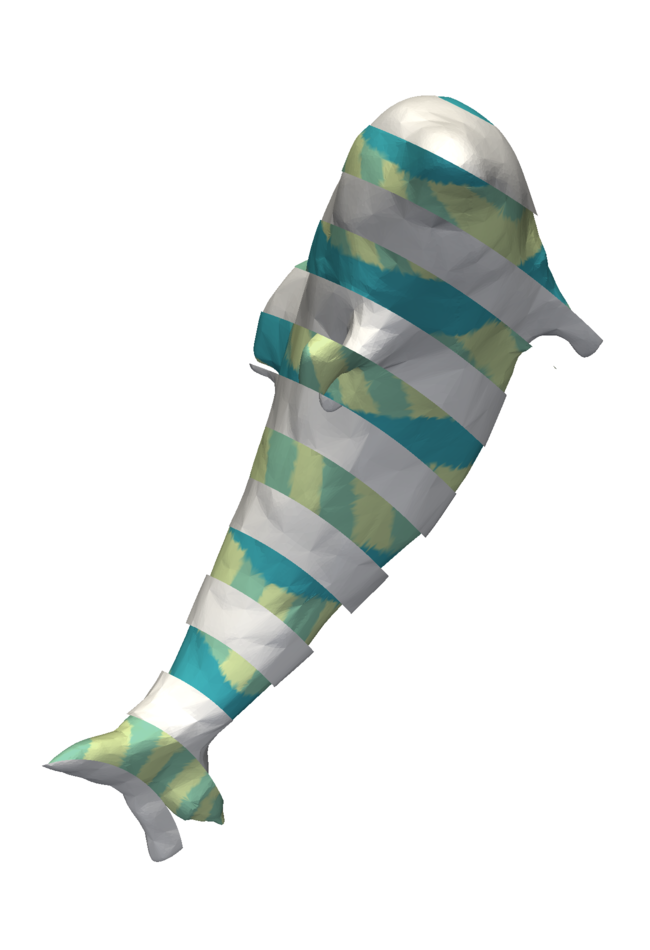}
  \hfill
  \includegraphics[width=.192\linewidth]{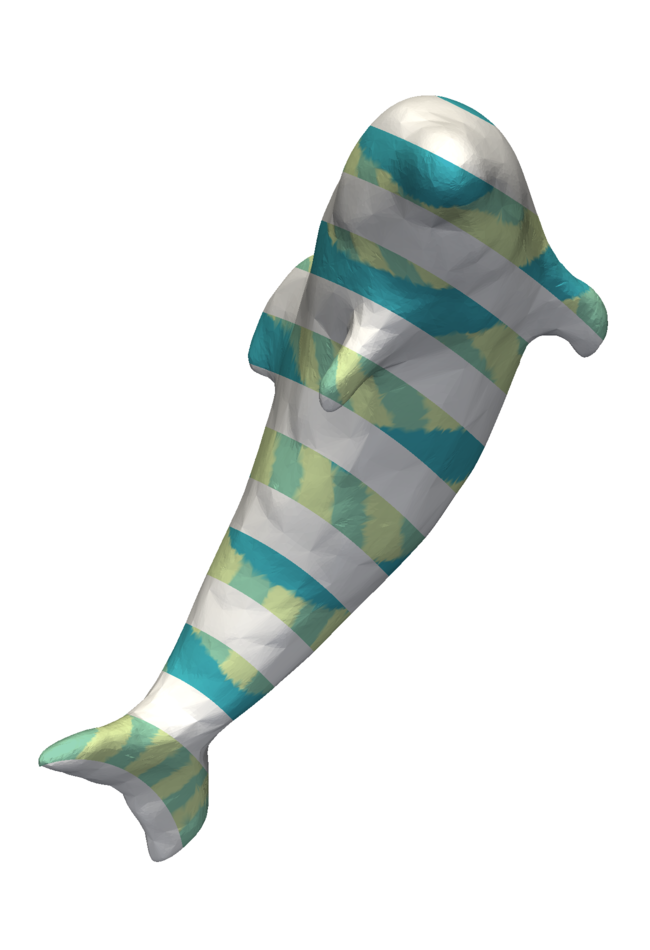}
  \hfill \mbox{}
  \caption{\label{fig:dolphin}From left to right: Textured dolphin $\M_1$, $\M_2$, resulting deformed shape $\phi(\M_1)$ after level $8$ in the minimization scheme, comparison of target and obtained shapes after the computation on grid level $4$ and $8$, respectively. The corresponding final grid is depicted in Figure \ref{fig:grids}.}
\end{figure*}
\begin{figure*}
  \centering
  \mbox{} \hfill
  \includegraphics[width=.24\linewidth]{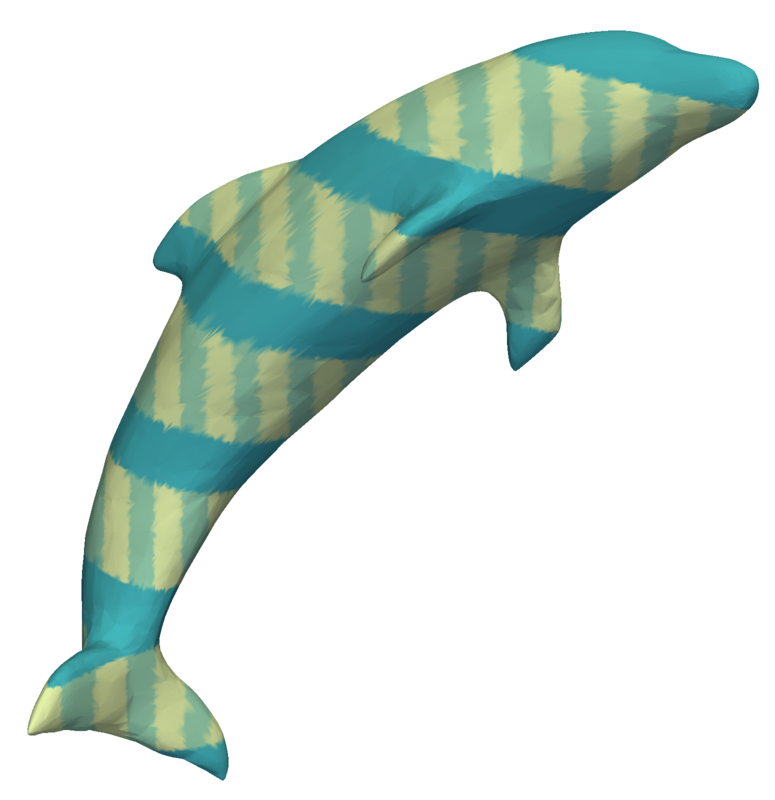}
  \hfill
  \includegraphics[width=.24\linewidth]{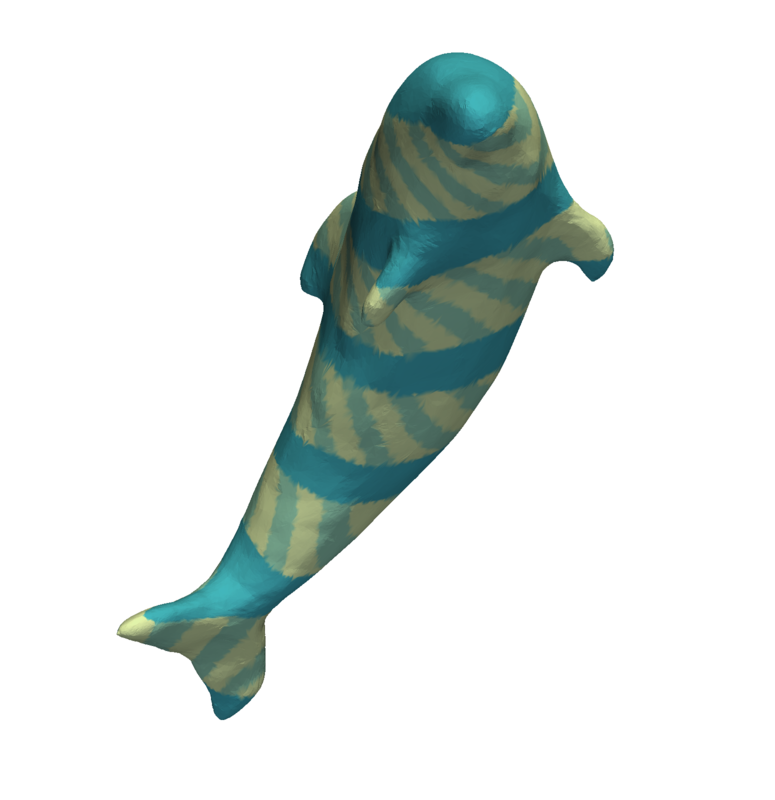}
  \hfill
  \vline
  \hfill
  \includegraphics[width=.24\linewidth]{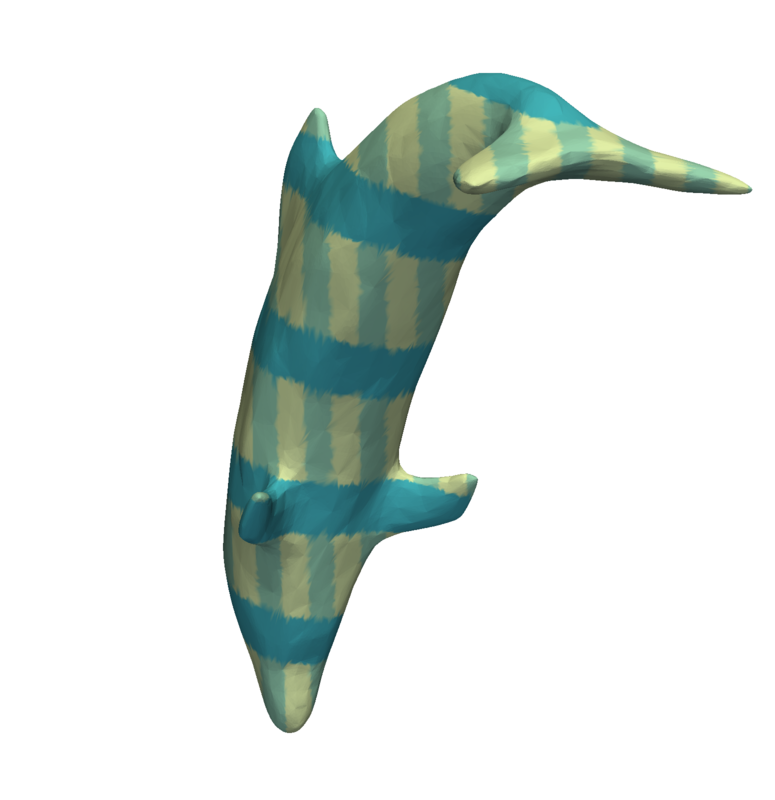}
  \hfill
  \includegraphics[width=.24\linewidth]{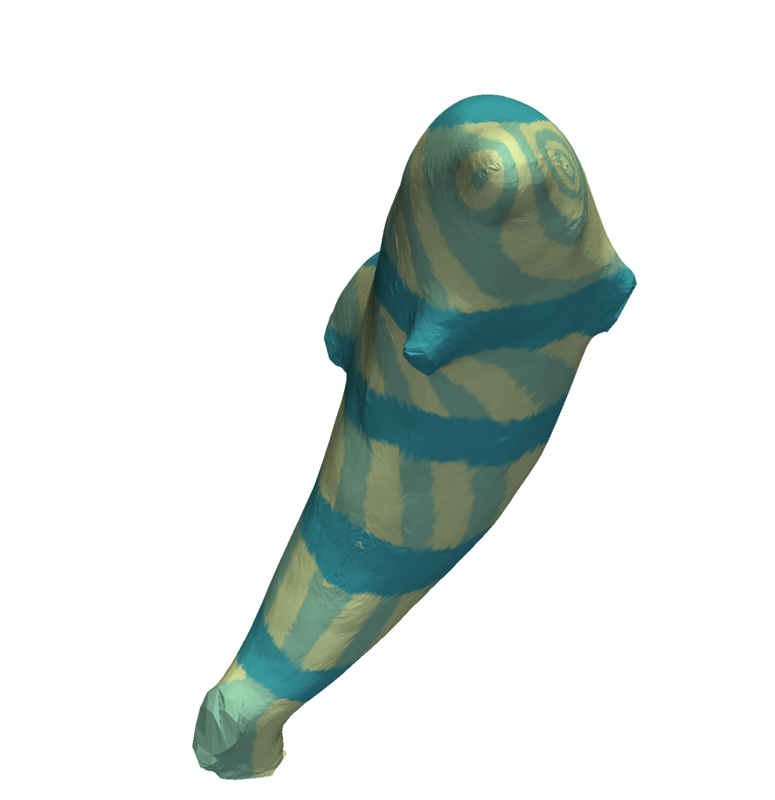}
  \hfill \mbox{}
  \caption{\label{fig:dolphinRot}From left to right: Initial shape of Figure \ref{fig:dolphin} after undergoing a rotation of $\pi/6$, deformed shape after level $8$ in the minimization (correct matching), after a rotation of $\pi$, and corresponding result (incorrect matching). Moderate changes in the initial alignment are handled correctly, while drastic ones are not.}
\end{figure*}
\begin{figure*}
  \centering
  \mbox{} \hfill
  \includegraphics[width=.23\linewidth]{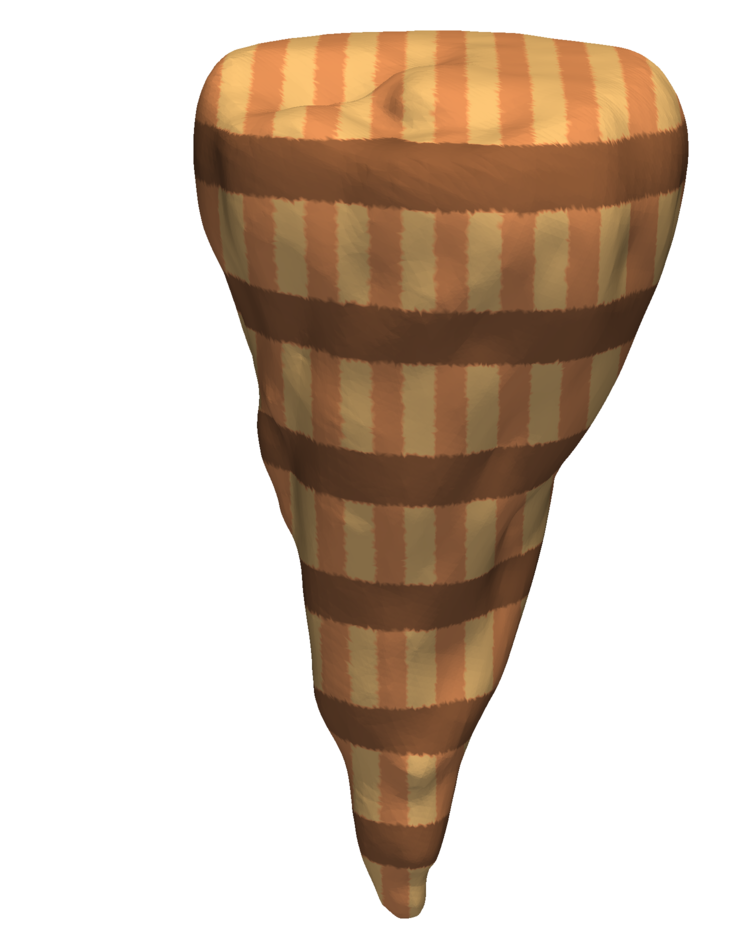}
  \hfill
  \includegraphics[width=.23\linewidth]{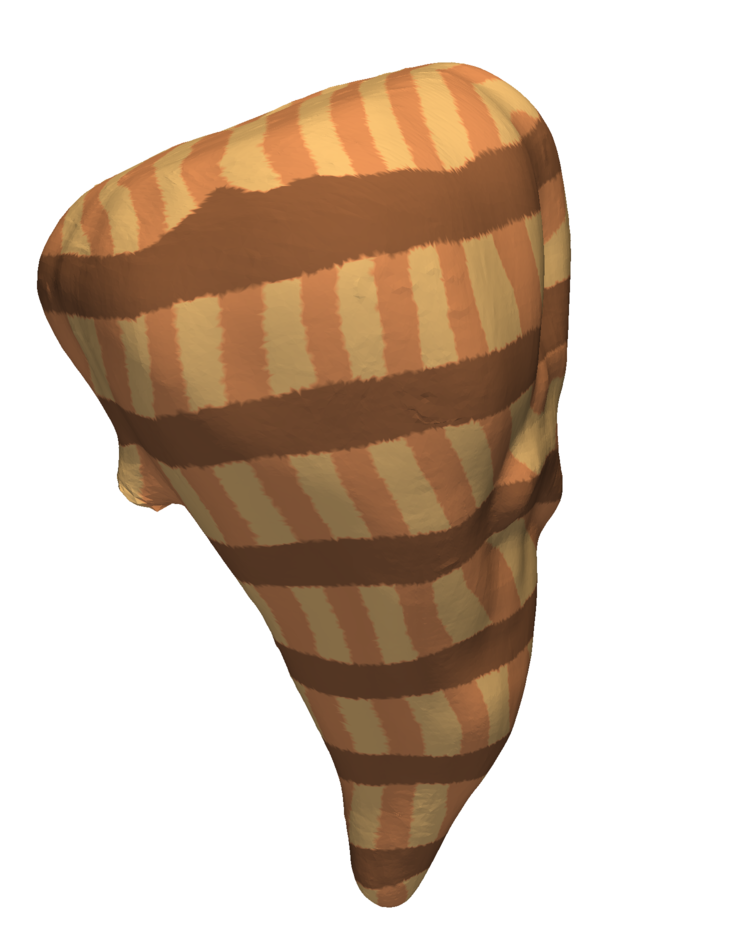}
  \hfill
  \includegraphics[width=.23\linewidth]{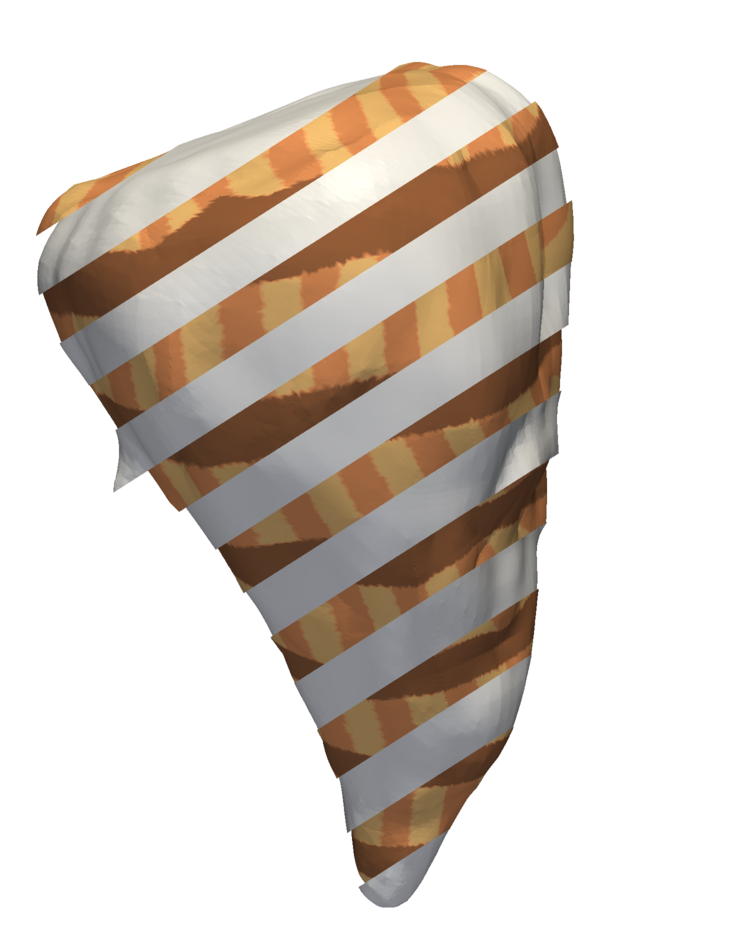}
  \hfill
  \includegraphics[width=.23\linewidth]{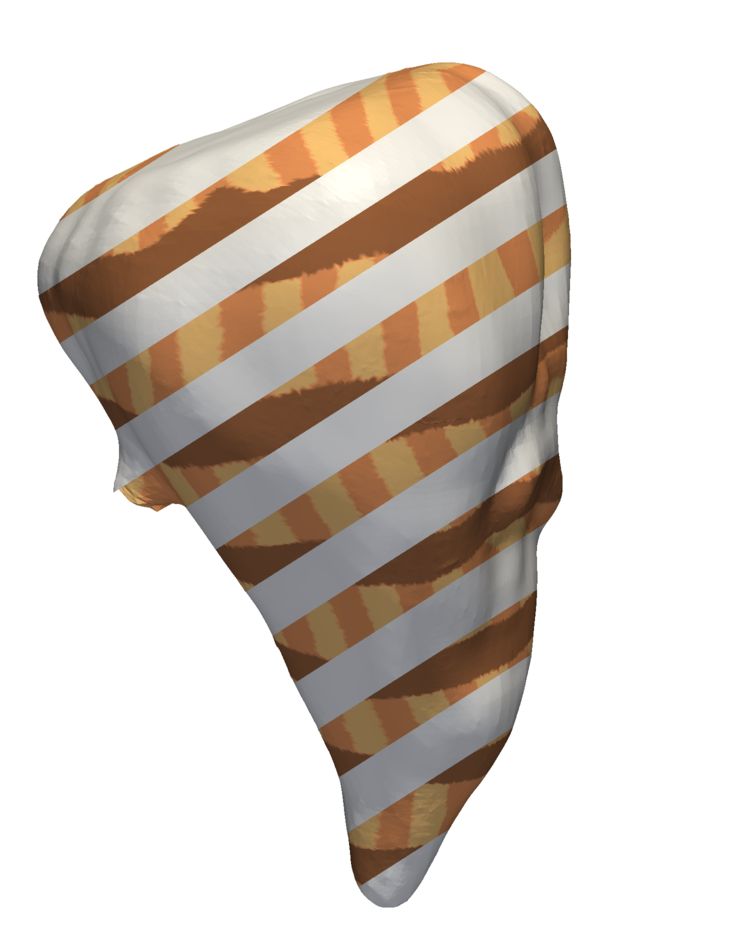}
  \hfill \mbox{}
  \caption{\label{fig:beets}From left to right: Textured sugar beet shape $\M_1$, resulting deformed shape $\phi(\M_1)$ after level $8$ in the minimization scheme, comparison of target sugar beet shape 
  and obtained shapes after the computation on grid level $4$ and $8$, respectively.}
\end{figure*}
\begin{figure*}
  \centering
  \mbox{} \hfill
  \raisebox{-0.5\height}{\includegraphics[width=.190\linewidth]{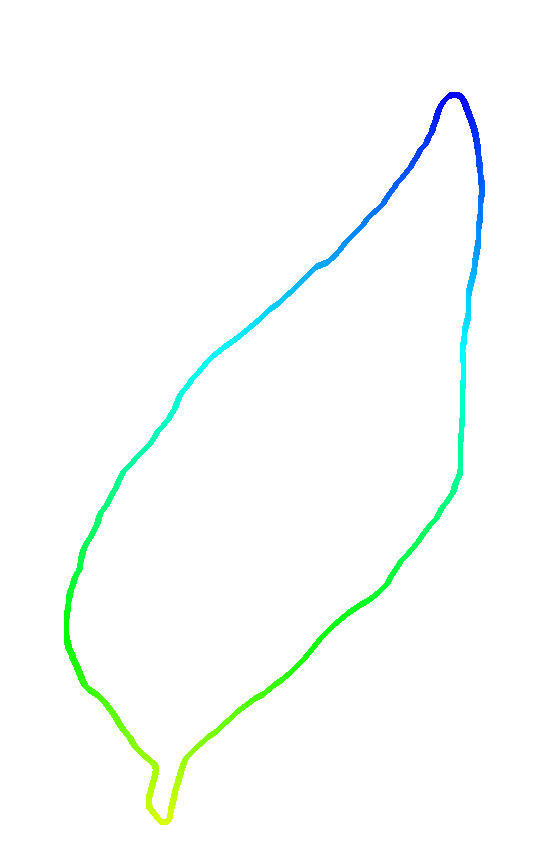}}
  \hfill
  \raisebox{-0.5\height}{\includegraphics[width=.190\linewidth]{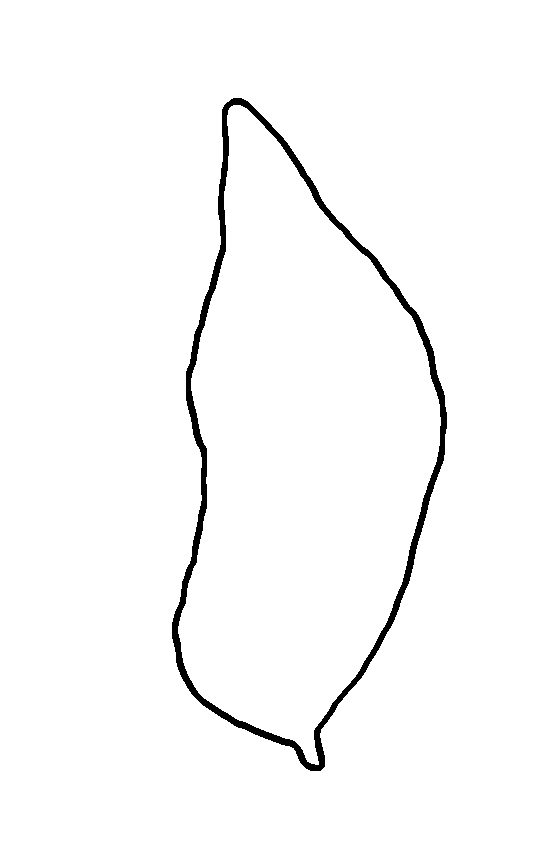}}
  \hfill
  \raisebox{-0.5\height}{\includegraphics[width=.190\linewidth]{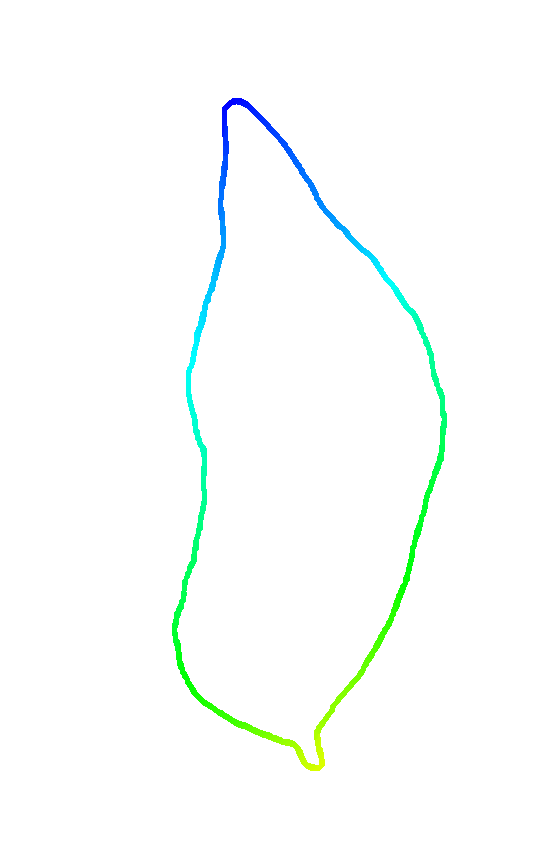}}
  \hfill
  \raisebox{-0.5\height}{\includegraphics[width=.400\linewidth]{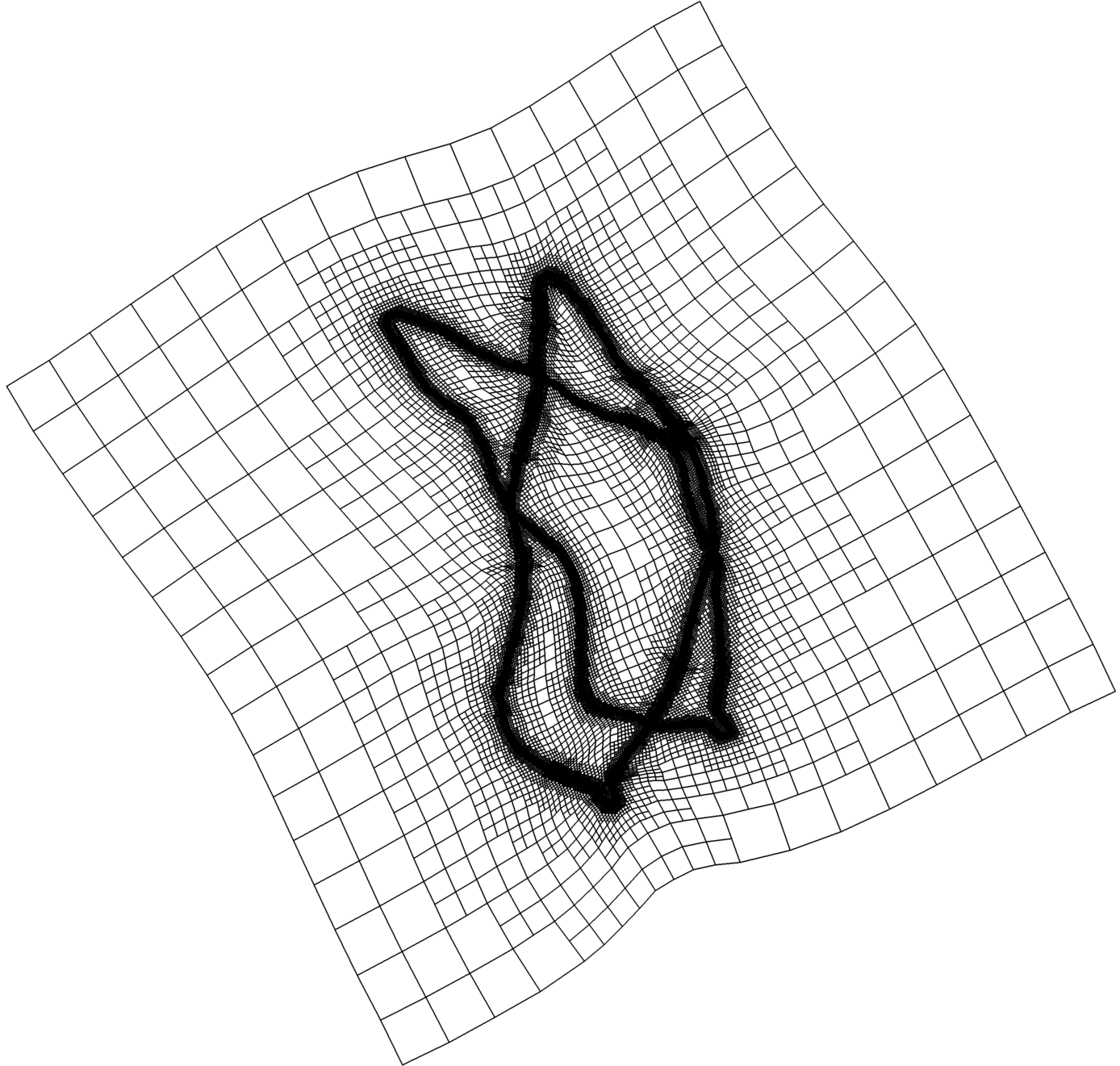}}
  \hfill \mbox{}
  \caption{\label{fig:leaves}2D example. From left to right: Colored leaf contours $\M_1$, $\M_2$, resulting deformed leaf shape $\phi(\M_1)$ after after the computation on grid level $10$ and corresponding deformed grid. The corresponding undeformed grid is depicted in Figure \ref{fig:grids}.}
\end{figure*}

%

\section*{Acknowledgements}
This research was supported by the Austrian Science Fund (FWF) through the National Research Network `Geometry+Simulation' (NFN S117) and Doctoral Program `Dissipation and Dispersion
in Nonlinear PDEs' (W1245). Furthermore, the authors acknowledge support of the Hausdorff Center for Mathematics at the University of Bonn. We would like to thank the anonymous reviewers for comments that have led to substantial improvements in this paper. The shapes for Figure \ref{fig:dolphin} are originally from the McGill 3D Shape Benchmark \cite{SidZhaMacShoBouDic08}. The scanned faces of Figure \ref{fig:faces} are part of the 3D Basel Face Model dataset \cite{PayKnoAmbVet09}. The laser-scanned sugar beets of Figure \ref{fig:beets} and the original shapes for Figure \ref{fig:hand} were kindly provided by Behrend Heeren.

\bibliographystyle{plain}
\bibliography{IgRuSc14}

\end{document}